\documentclass[11pt]{amsart}

\usepackage{amsmath}
\usepackage{amssymb}
\usepackage{amsthm}
\usepackage{graphicx} 
\usepackage{url}
\usepackage{hyperref}
\usepackage{multirow}
\usepackage{array}

\theoremstyle{plain}
\newtheorem{introtheorem}{Theorem}

\newtheorem{introcorollary}[introtheorem]{Corollary}

\theoremstyle{plain}
\newtheorem{theorem}{Theorem}[section]

\newtheorem{lemma}[theorem]{Lemma}

\theoremstyle{definition}
\newtheorem{definition}[theorem]{Definition}

\newtheorem{example}[theorem]{Example}
\newtheorem{question}[theorem]{Question}
\newtheorem{remark}[theorem]{Remark}

\def\term{\textbf}

\newcommand{\bord}{\partial}

\newcommand{\cro}{\mathrm{cr}}
\newcommand{\cros}{\mathrm{cr}}

\newcommand{\Dpqr}{\Delta_{p,q,r}}

\newcommand{\fgeod}{\varphi_{p,q,r}}

\renewcommand{\ge}{\geqslant}
\newcommand{\Gpqr}{G_{p,q,r}}

\newcommand{\Hopf}{\mathcal{H}^+_3}
\newcommand{\Hy}{\mathbb{H}^2}
\renewcommand{\H}{\mathrm{H}}

\newcommand{\lk}{\mathrm{lk}}
\renewcommand{\le}{\leqslant}

\newcommand{\nombre}[1]{\sharp_{#1}}

\newcommand{\PSLZ}{\mathrm{PSL}_2(\ZZ)}

\newcommand{\Qpqr}{{Q'_{p,q,r}}}
\newcommand{\Qred}{Q_{p,q,r}}
\newcommand{\QQ}{\mathbb{Q}}

\newcommand{\RR}{\mathbb{R}}

\newcommand{\Sph}{\mathbb{S}}
\newcommand{\Spqr}{\Hy/\Gpqr}

\newcommand{\Spn}{\Hy/G_{p_1, \dots, p_n}}

\newcommand{\template}{\mathcal T}
\newcommand{\Tpqr}{\template\!\!_{p,q,r}}

\newcommand{\TT}{\mathbb T}

\newcommand{\U}{\mathrm{T}^1}
\newcommand{\US}{\U\Spqr}

\newcommand{\ZZ}{\mathbb{Z}}

\title{Which geodesic flows are left-handed ?}

\author{Pierre Dehornoy}
\address{Univ. Grenoble Alpes, CNRS, Institut Fourier, F-38000 Grenoble, France
}
\email{pierre.dehornoy@univ-grenoble-alpes.fr}
\urladdr{\url{http://www-fourier.ujf-grenoble.fr/~dehornop/}}
\thanks{Many thanks to \'Etienne Ghys for raising the question addressed in this article and for numerous conversations. This article was written during a one-semester visit at the Laboratoire J.-V.\,Poncelet (CNRS UMI-2615) in Moscow. I thank the laboratoire and the CNRS for their support.}
\keywords{chirality, hyperbolic surface, knot, geodesic flow, Anosov flow, Markov partition, fibered knot, open book decomposition}
\subjclass[2010]{57M27, 37D40 (primary), and 37C27, 37D45, 57M20, 38B10, 37E35 (secondary)} 

\date{December 26th, 2014, corrected October 1st, 2016.}

\begin{document}

\maketitle

\begin{abstract}
We prove that the geodesic flow on the unit tangent bundle to a hyperbolic 2-orbifold is left-handed if and only if the orbifold is a sphere with three cone points. 
As a consequence, on the unit tangent bundle to a 3-conic sphere, the lift of every finite collection of closed geodesics is a fibered link.
\end{abstract}

%%%%%%%%%%%%%%%%%%%%%%%%%%%%%%%%%%%%%

\section{Introduction}
\label{S:Introduction}

Left-handed flows are a particular class of non-singular 3-dimensional flows on rational homology spheres introduced by \'Etienne Ghys~\cite{GhysLeftHanded}. 
This topological property roughly means that every pair of periodic orbits has negative linking number. 
It implies that every finite collection of periodic orbits bounds a so-called Birkhoff section (\emph{i.e.}, a global section with boundary) for the flow, and therefore forms a fibered link. 
In short, a left-handed flow can be written as an almost-suspension flow in as many ways as one can hope. 
The first examples of left-handed flows are the Seifert flows on~$\Sph^3$ and the Lorenz flow (although the second is not strictly speaking left-handed because of its fixed points). 
This gives an alternative proof that all torus links and all Lorenz links are fibered~\cite[Thm.\,5.2]{BW}. 
%It also gives a quick way to see that every link formed by an arbitrary finite number of fibers of the Hopf fibration on~$\Sph^3$ is a fibered link.

A natural question is then to look for other examples of left-handedness and to wonder whether such flows are abundant. 
The geodesic flow on the unit tangent bundle to any 2-dimensional sphere whose curvature is everywhere close to~$1$ is also left-handed~\cite{GhysLeftHanded}. 
This gives infinitely many examples of non-conjugated left-handed flows.
The geodesic flow on the unit tangent bundle to the modular surface~$\Hy/\PSLZ$ gives another example since its periodic orbits are isotopic to periodic orbits of the Lorenz flow~\cite[\S3.5]{Ghys}.
The cases of the almost-round spheres and of the modular surface lead to 

\begin{question}[Ghys]
	\label{Q:Ghys}
	Which geodesic flows are left-handed?
\end{question}

The goal of this article is to give a complete answer in the negatively curved case.
Strictly speaking, the unit tangent bundle to an orientable Riemanian surface is a 3-dimensional homology sphere if and only if the surface is a 2-sphere. 
But geodesic flows are naturally defined on a larger class, namely on 2-dimensional orbifolds, that is, surfaces locally modeled on a Riemannian disc or on the quotient of a disc by a finite rotation group---the so-called \emph{cone points} of the 2-orbifold (we restrict our attention here to orientable 2-orbifolds). 
In this larger class the unit tangent bundle is always a 3-manifold and it is a rational homology sphere if and only if the 2-orbifold is a 2-sphere with a finite number, say~$n$, of cone points---what we now call an \term{n-conic 2-sphere} (see Lemma~\ref{L:QHS} below). 
An $n$-conic 2-sphere admits a negatively curved metric if and only if $n\ge 3$ (in the case $n=3$ the orders of the cone points have to satisfy the additional constraint~$\frac 1 p+\frac 1 q+\frac 1 r<1$ and in the case $n=4$ the quadruple $(2,2,2,2)$ is prohibited). 
Since the geodesic flows associated to different negatively curved metrics are all topologically conjugated~\cite{Gromov}, one can speak of \emph{the} geodesic flow on a hyperbolic $n$-conic 2-sphere. 
Our main result is

\begin{introtheorem}
	\label{T:Main}
	Let $\Sigma$ be a hyperbolic $n$-conic 2-sphere.
	Then for $n=3$ the geodesic flow on~$\U\Sigma$ is left-handed and for $n\ge 4$ the geodesic flow is neither left-handed nor right-handed.
\end{introtheorem}

From the point-of-view of left-handedness, Theorem~\ref{T:Main} contains good and bad news. 
Good news is that it provides infinitely many new examples of left-handed flows on infinitely many different 3-manifolds. 
Bad news is that the answer to Question~\ref{Q:Ghys} is not as simple as one could hope. 
Indeed, a particular case of Theorem~\ref{T:Main} was proven in~\cite{Pierre}, namely that the geodesic flow on a 3-conic 2-sphere with cone points of order~$2,3,4g+2$ is left-handed. 
Also a historical construction of Birkhoff~\cite{Birkhoff} (generalized by Brunella, see~\cite[description 2]{Brunella}) implies that many collections of periodic orbits of the geodesic flow bound a surface that intersects negatively any other orbit of the geodesic flow, hence have negative linking number with any other periodic orbit. 
(The collections having this property are those that are symmetric, {\it i.e.}, such that if they contain the lift of an oriented geodesic they also contain the lift of the geodesic with the opposite orientation, and whose projections on the surface induce a checkerboard coloring of the complement.)
Thus the most optimistic conjecture was that the geodesic flow on \emph{any} hyperbolic $n$-conic 2-sphere is left-handed~\cite[Question 1.2]{Pierre}. 
Our present result states that this conjecture is false when the sphere has at least four cone points. 
%However we propose in the last section a weakened notion of left-handedness that is likely to hold for all geodesic flows.

As mentioned before, left-handedness implies that, in the complement of every finite collection of periodic orbits, there is a global section for the flow. 
This global section is a surface whose boundary is a multiple of the considered collection of periodic orbits: it is a \emph{Birkhoff section} for the flow~\cite{Fried, GhysLeftHanded}. 
Near the boundary of this surface, the flow induces a first-return map that is close to a rotation. 
If there is no mulitplicity (\emph{i.e.}, if the boundary of the section is exactly the considered collection of periodic orbits), then the first-return map is close to the identity near the boundary. 
Therefore we get an open book decomposition of the underlying 3-manifold (see~\cite{Rolfsen, Etnyre}). 
This restriction on the mulitplicy is achieved exactly when the integral homology class of the collection of periodic orbits is zero.
Hence Theorem~A directly implies

\begin{introcorollary}\label{Coro}
	For $\Sigma_{p,q,r}$ a 3-conic 2-sphere with hyperbolic metric, the lift in~$\U\Sigma_{p,q,r}$ of every finite collection of oriented geodesics on~$\Sigma_{p,q,r}$ whose class is zero in~$\H_1(\Sigma_{p,q,r}; \ZZ)$ is the binding of an open book decomposition of~$\U\Sigma_{p,q,r}$.
\end{introcorollary}

Since~$\U\Sigma_{2,3,7}$ is an integral homology sphere, Corollary~\ref{Coro} implies that every collection of orbits of the geodesic flow in~$\U\Sigma_{2,3,7}$ is the binding of an open book decomposition.

The proof of Theorem~\ref{T:Main} has two independent parts. 
The first part consists in proving that the geodesic flow on a hyperbolic $n$-conic 2-sphere with $n\ge 4$ is not left-handed. 
For this it is enough to find pairs of periodic orbits with linking number of arbitrary sign, and we do it by using an elementary construction.
The second part is more difficult and consists in proving that any two periodic orbits of the geodesic flow on a hyperbolic 3-conic 2-sphere has negative linking number.
Our proof heavily relies on the main result of~\cite{PierreTali} where we constructed a \emph{template} with two ribbons for the geodesic flow on every 3-conic 2-sphere. 
Using this template, we estimate the linking number of every pair of orbits, and prove that it is always negative. 

The plan follows the above scheme: in Section~\ref{S:Preliminaries}, we recall the necessary definitions, in particular what are \emph{geodesic flows} (\S~\ref{S:Spqr}) and \emph{left-handed flows} (\S~\ref{S:LeftHanded}). 
We prove the $n\ge 4$-part of Theorem~\ref{T:Main} in \S~\ref{S:n4}. 
In \S~\ref{S:Template} we present the template for the geodesic flow on a hyperbolic 3-conic 2-sphere constructed in~\cite{PierreTali}.
Section~\ref{S:Computation} then contains the proof of the $n=3$-part of Theorem~\ref{T:Main}.

\tableofcontents

%%%%%%%%%%%%%%%%%%%%%%%%%%%%%%%%%%%%%
\section{Preliminaries}
\label{S:Preliminaries}

%%%%%%%%%%%%%%%%%
\subsection{Orbifolds, unit tangent bundles, and geodesic flows}
\label{S:Spqr}

A \term{2-dimensional orientable orbifold} is a topological surface equipped with a metric that is locally isometric 
to the quotient of a Riemannian disc by a finite-order rotation group. 
It is \term{hyperbolic} if the metric has everywhere negative curvature. 
The \term{type} of the orbifold is $(g; p_1, \dots, p_n)$, where $g$ is the genus of the underlying surface and $p_1, \dots, p_n$ are the orders of the cone points of the orbifold. 

The {unit tangent bundle}~$\U D$ to a disc~$D$ equipped with a Riemannian metric is the set of tangent vector of length~$1$, hence it is homeomorphic to the solid torus~$D\times \Sph^1$. 
If a finite group~$\ZZ/k\ZZ$ acts by rotations on~$D$, then it acts faithfully on~$\U D$. 
The quotient $\U D/(\ZZ/k\ZZ)$ is then a 3-manifold. 
Actually it is also a solid torus which admits a Seifert fibration whose fibers are the fibers of the points of~$D/(\ZZ/k\ZZ)$.

The \term{unit tangent bundle to a 2-dimensional orbifold} is the 3-manifold that is locally modeled on the quotient of the unit tangent bundle to a Riemannian disc by a finite-order rotation group.
So it is a Seifert fibered space.

\begin{lemma}
	\label{L:QHS}
	The unit tangent bundle to an orientable 2-orbifold~$\Sigma$ of type~$(g; p_1, \dots, p_n)$ is a rational homology sphere if and only if $g=0$.
\end{lemma}

\begin{proof}
	If $\Sigma$ is not a topological 2-sphere, it contains a nonseparating simple closed curve. 
	The lift of this curve in~$\U\Sigma$ yields a non-trivial element of $\H_1(\U\Sigma; \QQ)$. 
	So if $\U\Sigma$ is a homology 3-sphere, then~$\Sigma$ is a topological 2-sphere.

	Conversely, if $\Sigma$ is of type $(0;p_1, \dots, p_n)$ its unit tangent bundle is the Seifert fibered space 
	with presentation $(Oo0\,\vert\, 2{-}n; $ $(p_1, 1), \dots, (p_n, 1))$ in the notation of~\cite[p.\,140]{Montesinos} (see  p.\,183 for a proof). 
	By~\cite[p.\ 4]{Saveliev} it is a $\QQ$-homology sphere.
\end{proof}
	
\begin{remark}
	We can also use \cite[p.\ 4]{Saveliev} to see that $\U\Sigma_{0;p_1, \dots, p_n}$ can be a $\ZZ$-homology sphere only for $n=3$. 
	Indeed, the order of $\H_1(\U\Sigma_{0;p_1, \dots, p_n}; \ZZ)$ is $\vert (n-2)p_1\dots p_n - \sum_i p_1\dots \hat p_i\dots p_n\vert$. 
	The condition $\vert (n-2)p_1\dots p_n - \sum_i p_1\dots \hat p_i\dots p_n\vert = 1$ then implies that the $p_i$'s must be pairwise coprime 
	and, by the pigeonhole principle, that one of them must be smaller than $\frac n{n-2}$. 
	For $n=3$, we find the two solutions $(2,3,5)$---which corresponds to Poincar\'e dodecahedral space--- and $(2,3,7)$. 
	For $n\ge 4$, these conditions cannot be fulfilled.
	So Theorem~\ref{T:Main} concerns only one integral homology sphere.
\end{remark}

Thurston showed~\cite{ThurstonNotes} that an orientable 2-orbifold of type $(0; p_1, \dots, p_n)$ admits a hyperbolic metric 
if and only if~$n\ge 5$ or $n=4$ and $(p_1, \dots, p_4)\neq(2,2,2,2)$ or $n=3$ and $\frac 1 {p_1} + \frac 1 {p_2} + \frac 1 {p_3} <1$. 
In this case the orbifold is covered by the hyperbolic plane and it is isometric to~$\Hy / G_{p_1, \dots, p_n}$ for some Fuchsian group~$G_{p_1, \dots, p_n}$. 

For every unit tangent vector to~$\Hy$, there exists a unique geodesic oriented by this tangent vector, 
so that every point of~$\U\Hy$ can be written in a unique way in the form~$(\gamma(0), \dot\gamma(0))$ 
where $\gamma$ is a geodesic travelled at speed~$1$. 
The \term{geodesic flow}~$\varphi$ on~$\U\Hy$ is then defined by~$\varphi^t(\gamma(0), \dot\gamma(0)) := (\gamma(t), \dot\gamma(t))$. 
Since every Fuchsian group~$G$ acts by isometries on~$\Hy$, we can define the geodesic flow on~$\U\Hy/G$ by modding out. 
The important property for our purpose is that an orbit of the geodesic flow is the lift of an oriented geodesic, 
and therefore that periodic orbits of the geodesic flow are lifts of oriented closed geodesics.

%%%%%%%%%%%%%%%%%
\subsection{Linking number in homology spheres}
\label{S:Linking}

Given two disjoint oriented closed curves $\gamma_1, \gamma_2$ in a closed 3-manifold whose rational homology classes are trivial (in particular any disjoint closed curves if $M$ is a rational homology sphere), their \term{linking number} $\lk_M(\gamma_1, \gamma_2)$ is defined as the intersection number of~$\gamma_1$ with a (rational) 2-chain bounded by~$\gamma_2$. 
It is a rational number whose denominator divides the order of the torsion part of~$\H_1(M; \ZZ)$. 
The nullity of $[\gamma_1]\in H_1(M; \QQ)$ implies that this intersection number does not depend on the choice of the 2-chain. 
The same definition extends for $\gamma_1, \gamma_2$ two homologically trivial finite collections of oriented closed curves. 

\begin{example} 
\label{Ex}
	Assume that $\Sigma$ is a genus~$g$-surface with~$g\ge2$ and $f_1, f_2$ are the trigonometrically oriented fibers of two generic points~$x_1, x_2$. 
	Since $\chi(\Sigma)$ is non-zero, there exists a vector field~$Y_1$ on~$\Sigma$ with only one singularity and we can assume that this singularity is at~$x_1$. 
	By the definition of Euler characteristics, the singularity has index~$\chi(\Sigma)$. 
	So~$Y_1$ lifts in~$\U\Sigma$ to a 2-chain whose boundary is~$-\chi(\Sigma)\,f_1$. 
	This implies that $f_1$ is homologically trivial. 
	Moreover, since the lift of~$Y_1$ has intersection~$+1$ with any other generic fiber, in particular with~$f_2$, we have $\lk_{\U\Sigma}(-\chi(\Sigma)f_1, f_2)=+1$, and~$\lk_{\U\Sigma}(f_1, f_2)= -1/{\chi(\Sigma)}$. 

	When $\Sigma$ is any orientable 2-orbifold we also have~$\lk_{\U\Sigma}(f_1, f_2)= -1/{\chi(\Sigma)}$. The proof is similar, except that we have to consider a multivalued vector field. 
\end{example}

%%%%%%%%%%%%%%%%%
\subsection{Left-handed and Anosov flows}
\label{S:LeftHanded}

We now recall the notion of \emph{left-handed flow}, based on~\cite{GhysLeftHanded}. 
The reader in a hurry can directly take Lemma~\ref{L:Anosov} as a definition, since it is all we need in the sequel. 

Roughly speaking, a 3-dimensional flow in a $\QQ$-homology sphere is left-handed if all pairs of periodic orbits are negatively linked. 
However taking this as a definition would produce some strange results as, for example, a flow with no periodic orbit would be left-handed. 
The precise definition actually involves invariant measures, which can be seen as generalizations of periodic orbits 
(indeed a periodic orbit induces a canonical invariant measure: the linear Dirac measure whose support coincide with the periodic orbit). 
Invariant measures form a non-empty convex cone. 

A \term{Gauss\ linking form} on a $\QQ$-homology sphere~$M$ is a $(1,1)$-form on $C(2,M)$---the configuration space of pairs of disjoint points--- whose integral on the product of two disjoint curves gives their linking number. 
Gauss linking forms always exist (Gauss gave the first example on~$\RR^3$, see also~\cite{DTG} for explicit examples on~$\Sph^3$ and $\mathbb{H}^3$ and~\cite{Lescop} for a construction on an arbitrary $\QQ$-homology sphere). 
Given a vector field~$X$ on~$M$, the linking number of two $X$-invariant measures~$\mu, \mu'$ not charging the same periodic orbits is then defined 
by $\lk_{M,X}(\mu, \mu'):=\iint \omega(X(x), X(y)) d\mu(x)d\mu(y)$, where $\omega$ is any diffuse Gauss linking form (the integral always converges).

\begin{definition}\cite{GhysLeftHanded}
\label{D:LeftHanded}
	A non-singular vector field $X$ on a $\QQ$-homology sphere~$M$ is \term{left-handed} 
	if for every invariant measures $\mu, \mu'$ not charging the same periodic orbits we have $\lk_{M,X}(\mu, \mu')<0$.
\end{definition}

Our goal here is not to paraphrase~\cite{GhysLeftHanded}. 
Let us just repeat that left-handed flows have very nice topological properties. 
In particular every finite collection of periodic orbits forms a fibered link whose fiber surfaces are Birkhoff sections for the flow (\emph{i.e.}, intersect every orbit with a bounded first-return time).

In general the space of invariant measures of a vector field is huge (infinite dimensional) and hard to determine. 
However when $X$ is of Anosov type ---as in particular the geodesic flow on a negatively curved 2-orbifold--- left-handedness 
reduces to a property of periodic orbits. 
The reason is the shadowing property, 
namely that every invariant measure is the weak limit of a sequence of (Dirac measures supported by) periodic orbits.

\begin{lemma}\cite[Lemma 2.1]{Pierre}
	\label{L:Anosov}
	A non-singular Anosov vector field $X$ on a $\QQ$-homology sphere is left-handed if and only if every pair of periodic orbits of~$X$ has negative linking number.
\end{lemma}

%%%%%%%%%%%%%%%%%
\subsection{2-spheres with at least four cone points}
\label{S:n4}

We now prove the elementary part of Theorem~\ref{T:Main}, 
namely that the geodesic flow on~$\U\Spn$ is neither left-handed nor right-handed for $n\ge 4$. 
Since the geodesic flows corresponding to different hyperbolic metrics on the same orbifold are topologically conjugated, 
we can choose our preferred metric.
We then choose the metric so that $\Spn$ is the union of two $n$-gons $F_1, F_2$ in~$\Hy$ 
with angles $\pi/p_1, \dots, \pi/p_n$ glued along their boundaries. 
The polygons $F_1, F_2$ have $n$ vertices $P_1, \dots, P_n$ and are then images one from the other in a mirror (see Fig.~\ref{F:n>4}).

\begin{lemma}
	\label{L:NonIntersecting}
	With the above metric, there exists two periodic geodesics on~$\Spn$ that do not intersect.
\end{lemma}

\begin{proof}
	For $n\ge 5$ (see Fig.~\ref{F:n>4} left), it is enough to choose five consecutive sides $e_1, \dots, e_5$ of~$F_1$, to consider~$s_1$ the shortest segment connecting $e_1$ to~$e_3$ and $s_2$ the shortest segment connecting~$e_3$ to~$e_5$. 
	The two segments $s_1, s_2$ do not intersect on~$F_1$ and their symmetrics do not intersect on~$F_2$ as well. 
	The union of~$s_1$ with its symmetric on~$F_2$ yields a closed geodesic on~$\Spn$ that does not intersect the union of $s_2$ with its symmetric on~$F_2$.
	
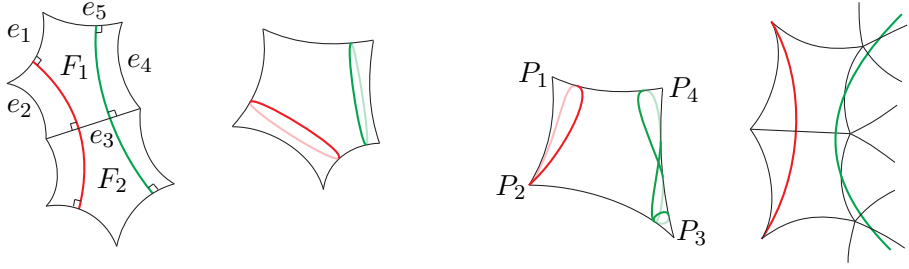
\begin{figure}
	\begin{picture}(120,35)(0,0)
	\put(0,0){\includegraphics[width=.8\textwidth]{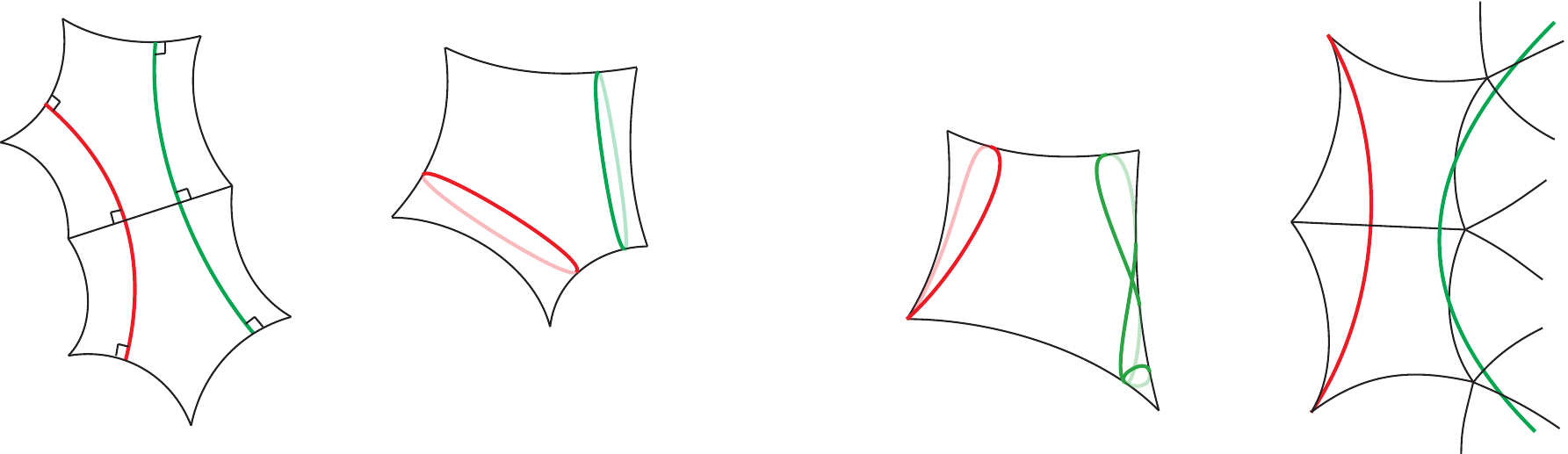}}
	\put(0,30){$e_1$}
	\put(00,20){$e_2$}
	\put(11,16){$e_3$}
	\put(16,26){$e_4$}
	\put(10,33){$e_5$}
	\put(7,25){$F_1$}
	\put(12,10){$F_2$}
	\put(68,24){$P_1$}
	\put(65,9){$P_2$}
	\put(88,22){$P_4$}
	\put(89,3){$P_3$}
	\end{picture}
	\caption{\small Two non-intersecting geodesics on~$\Spn$ for $n>4$ on the left and $n=4$ on the right (on this picture all corners but the bottom left have odd order).}
	\label{F:n>4}
\end{figure}
	
	For $n=4$, we have to refine the idea (see Fig.~\ref{F:n>4} right). 
	We will find two disjoint closed geodesics~$g_{12}, g_{34}$ which are close to the sides of~$P_1P_2$ and $P_3P_4$ respectively.  
	First suppose that the orders $p_1$ and $p_2$ are both even, then the side $P_1P_2$ actually supports a closed geodesics that we choose for~$g_{12}$. 
	Now if $p_1$ is odd and $p_2$ is even, there is a geodesic starting at~$P_2$ and winding $\frac{p_1-1}2$ times around~$P_1$ before coming back to~$P_2$, and then makes the trip in the opposite direction (as along $P_1P_2$ on Fig.~\ref{F:n>4} right). 
	We choose it for~$g_{12}$. 
	Finally if $p_1, p_2$ are both odd, then there is a closed geodesic that starts on the edge $P_1P_2$, winds $\frac{p_1-1}2$ times around~$P_1$, comes back to its initial point, then winds $\frac{p_2-1}2$ around~$P_2$ and comes back to its initial points with its initial direction (as along $P_3P_4$ on Fig.~\ref{F:n>4} right). 
	We choose it for~$g_{12}$.
	Applying the same strategy for~$g_{34}$ we check that, since $g_{12}$ stays close to~$P_1P_2$ and $g_{34}$ stays close to~$P_3P_4$, they do not intersect. 
	(Of course, the same strategy works also for $n\ge 5$, but we had an easier construction in that case.)
\end{proof}

\begin{proof}[Proof of Theorem~\ref{T:Main} for $n\ge 4$]
	Consider $\gamma_1, \gamma_2$ two non-intersecting closed geodesics on $\Spn$ (they exist by Lemma~\ref{L:NonIntersecting}).
	Denote by $\overset{\leftrightarrow}{\gamma_1}$ the set in $\U\Spn$ of those unit vectors that are tangent to~$\gamma_1$ regardless of the orientation, we call it the \term{symmetric lift} of~$\gamma_1$. 
	It is a 2-component link if~$\gamma_1$ does not visit a cone point of even order, and a knot otherwise.
%	(Equivalently when~$\gamma_1$ does not visit a cone point of even order, 
%	we choose a unit-speed parametrization of~$\gamma_1$ and consider the set of speed vectors of~$\gamma_1$ plus the set of speed vectors of~$-\gamma_1$.) 
	Now consider the set $S_{\overset{\leftrightarrow}{\gamma_1}}$ of all unit tangent vectors based on points of~$\gamma_1$ (see Fig.~\ref{F:n>4surface}).
	This is the union of two immersed annuli that we orient so that the boundary of the integral 2-chain $S_{\overset{\leftrightarrow}{\gamma_1}}$ is~$2\overset{\leftrightarrow}{\gamma_1}$.
		Since~$S_{\overset{\leftrightarrow}{\gamma_1}}$ lies only in the fibers of the point of~$\overset{\leftrightarrow}{\gamma_1}$ and since $\gamma_2$ does not intersect~$\gamma_1$, the intersection of $\overset{\leftrightarrow}{\gamma_2}$ with $S_{\overset{\leftrightarrow}{\gamma_1}}$ is empty. 
	Therefore $\lk_{\U\Spn}(\overset{\leftrightarrow}{\gamma_1},\overset{\leftrightarrow}{\gamma_2})=0$.

\begin{figure}
	\includegraphics[width=.2\textwidth]{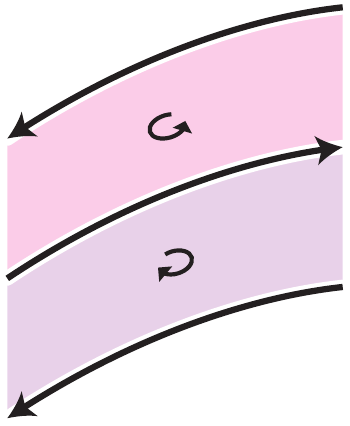}
	\caption{\small The surface $S_{\overset{\leftrightarrow}{\gamma_1}}$ consists of two annuli. Its oriented boundary is $2\overset{\leftrightarrow}{\gamma_1}$.}
	\label{F:n>4surface}
\end{figure}

\end{proof}

\begin{remark}
The reader may be frustrated that in the above proof we exhibited only pairs of periodic orbits 
with linking number zero and not with positive linking number. 
Actually, this is only to simplify the presentation. 
Indeed, if $\gamma_1, \gamma_2$ are two disjoint oriented geodesics on~$\Spn$, 
then $\overset{\rightarrow}{\gamma_1}$ is homologous in the complement of~$\overset{\rightarrow}{\gamma_2}$ to the multiple of some fibers, say $\lambda_1 f_1$, and similarly $\overset{\rightarrow}{\gamma_2}$ is homologous to some $\lambda_2 f_2$.
Since the linking number of two regular fibers is~$- 1 /\chi$, we have $\lk_{\U\Spn}(\overset{\rightarrow}{\gamma_1}, \overset{\rightarrow}{\gamma_2}) = - \lambda_1\lambda_2/\chi$. 
In the proof of Theorem~\ref{T:Main}, we had chosen the geodesics $\gamma_1, \gamma_2$ so that $\lambda_1=\lambda_2= 0$. 
By adding some winding around the cone points, 
we can make $\lambda_1, \lambda_2$ arbitrarily large in the positive or in the negative direction, keeping~$\gamma_1$ and $\gamma_2$ disjoint. 

However, notice that if $\gamma_1$ and $\gamma_2$ intersect, the symmetry between positive and negative is broken, 
as intersection points add a negative contribution to the linking number (namely 
a $J^-$-move~\cite{Arnold} on~$\gamma_1$ or $\gamma_2$ adds $-1$ to~$\lk_{\U\Spn}({\overset{\rightarrow}{\gamma_1}}, \overset{\rightarrow}{\gamma_2})$). 
This explains why linking numbers of lifts of long geodesics are likely to be negative.
\end{remark}
%Here is a lemma that helps understanding liking numbers in~$\U\Spn$ and then allows to construct pairs of geodesic with positive linking number.
%
%The \term{writhe}~$\wri(\gamma)$ of an immersed curve~$\gamma$ in the plane is defined as the index of the map $\frac{\gamma'}{\vert\vert\gamma'\vert\vert}:\Sph^1\to\Sph^1$. 
%It is an integer that is invariant under homotopy among immersions of curves. 
%Actually, Whitney showed that every immersed curve~$\gamma$ in the plane can be homotoped on a circle travalled~$n$ times. 
%In this case one has $n=\wri(\gamma)$.
%Suppose we fix a point~$*$ in~$\Spn$ and that~$\gamma$ is an immersed curve in~$(\Spn)\setminus\{*\}$. 
%We can also define a \term{writhe} for~$\gamma$ that will be a rational number in the following way: 

%%%%%%%%%%%%%%%%%
\subsection{The template~$\Tpqr$ and its extremal orbits}

Now we turn to orbifolds of type $\Hy/G_{p_1, p_2, p_3}$, that we prefer to denote by~$\Spqr$, 
and we denote by~$\fgeod$ the geodesic flow on~$\U\Spqr$. 
We first recall two results of~\cite{PierreTali} that describe the isotopy class of all periodic orbits of~$\fgeod$.

\begin{lemma}\cite[Prop.\,2.4]{PierreTali}
	\label{L:Hopf}
	The unit tangent bundle~$\U\Spqr$ is obtained from~$\Sph^3$ by surgeries of index~$p{-}1, q{-}1, r{-}1$ 
	on the three components of a positive Hopf link.
\end{lemma}

A \term{template} (see~\cite{BW, GHS}) in a 3-manifold is an embedded branched surface made of finitely many rectangular ribbons that are glued along their horizontal sides in such a way that every gluing point is on the bottom side of at most one ribbon (but may be on the top side of several ribbons). 
A template is equipped with the vertical bottom-to-top flow on each rectangle.  
This is actually only a semi-flow since orbits in negative time are not uniquely defined when crossing a branching line.
Given a labeling of the ribbons of the template, the \term{code} of an orbit is the infinite sequence that describes the consecutive ribbons used by the orbit. 
The \term{kneading sequences} (see~\cite{dMvS, HS}) are the codes of the leftmost and rightmost orbits of every ribbon.
(Sometime~\cite{BW, GHS} it is required that the gluing map be Markovian, namely that the top side of every ribbon be glued to the union of the full bottom sides of several ribbons. 
Here we remove this condition, this is in a sense the price to pay for having  a template with only two ribbons.)

Call~$\Tpqr$ the template with two ribbons whose embedding in~$\US$ is depicted on Figure~\ref{F:Tpqr}, whose left and right ribbons are labelled by $a$ and $b$ respectively, 
and whose kneading sequences are the words $u_L, u_R, v_L, v_R$ given by Table~\ref{Table}.

\begin{figure}[ht]
  	\begin{picture}(90,70)(0,0)
      		\put(0,0){\includegraphics*[width=.6\textwidth]{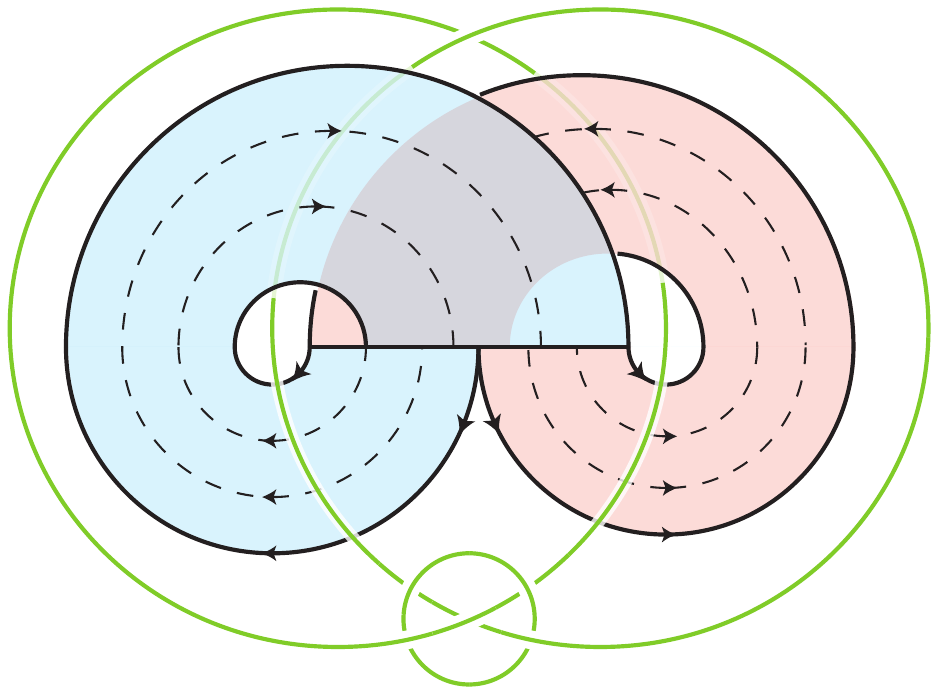}}
      		\put(-4,46){$p{-}1$}
      		\put(88,46){$q{-}1$}
      		\put(49,0){$r{-}1$}
      		\put(18.3,35.5){$u_L$}
      		\put(67.5,35.5){$v_R$}
      		\put(51,21){$v_L$}
      		\put(36.5,21){$u_R$}
      		\put(18,51){$a$}
      		\put(68,51){$b$}
  	\end{picture}
  	\caption{\small The template~$\Tpqr$ in~$\US$. 
  	The 3-manifold~$\US$ is obtained from~$\Sph^3$ by surgeries on a three-components Hopf link (green) with the given indices. 
  	The template~$\Tpqr$ is characterized by its embedding in~$\US$ and by the kneading sequences that describe the orbits of the extremities of the ribbons. 
  	In case $r$ is infinite, $\US$ is the open manifold obtained by removing the bottommost component of the link.
	In the case $p=2$, the exit side of the left ribbon is strictly included into the entrance side of the right ribbon.}
  	\label{F:Tpqr}
\end{figure}

\begin{table}[ht]
	\begin{center}
	\begin{tabular}{c|c|c}
		& $u_L$ & $v_R$ \\ \hline
   		$r$ infinite & $(a^{p-1}b)^\infty$  
  	  	& $(b^{q-1}a)^\infty$ \\ \hline
		\hline
		$p, q, r>2$& & \\ \hline
  		$r$ odd & $((a^{p-1}b)^{\frac{r-3}{2}} a^{p-1}b^2)^\infty$  
  	  	& $((b^{q-1}a)^{\frac{r-3}{2}} b^{q-1}a^2)^\infty$ \\ \hline
		$r$ even & $((a^{p-1}b)^{\frac{r-2}{2}} a^{p-2}(ba^{p-1})^{\frac{r-2}{2}}b^2)^\infty$
		& $((b^{q-1}a)^{\frac{r-2}{2}} b^{q-2}(ab^{q-1})^{\frac{r-2}{2}}a^2)^\infty$ \\ \hline
		\hline
		$p=2, q>2,r>4$& & \\ \hline
   		$r$ odd & $((ab)^{\frac{r-3}{2}} ab^2)^\infty$  
  	  	& $b^{q-1}((ab^{q-1})^{\frac{r-5}{2}} ab^{q-2})^\infty$ \\ \hline
  		$r$ even & $((ab)^{\frac{r-4}{2}} ab^2)^\infty$  
  	  	& $b^{q-1}((ab^{q-1})^{\frac{r-4}{2}} ab^{q-2})^\infty$ \\ \hline
	\end{tabular}
		
	\vspace{1mm}
	 $v_L:=bu_L,$ 
	 
	 $u_R:=av_R$
	 \end{center} 
	\caption{\small The kneading sequences of the template~$\Tpqr$. }
	\label{Table}
\end{table}

\begin{theorem}\cite[Theomem A]{PierreTali}
  	\label{T:Template}
	Up to one exception, there is a one-to-one correspondence between periodic orbits of the geodesic flow~$\fgeod$ on~$\US$ and periodic orbits of the template~$\Tpqr$ such that its restriction to every finite collection is induced by an isotopy (which depends on the collection).
\end{theorem}

The exception mentioned in the statement consists of two orbits of the template that correspond to the same orbit of the geodesic flow. 
In the sequel, we only need to consider the knots and links in $\Tpqr$ since our strategy is now to prove that any orbits of~$\Tpqr$ have negative linking number in~$\U\Spqr$. 
The exception is therefore not a problem as we will just prove twice that a given orbit links negatively with all others.

%%%%%%%%%%%%%%%%%
\subsection{Computing linking numbers in~$\U\Spqr$}
\label{S:Template}

Our goal is to estimate linking numbers of orbits of the template~$\Tpqr$. 
The latter sits in~$\U\Spqr$, but is depicted in~$\Sph^3$.
Therefore we need a formula that gives the linking number after surgery in terms of information that can be read directly on Figure~\ref{F:Tpqr}, 
namely the linking number before surgery and the linking numbers of the links with the different components of the Hopf link~$\Hopf$.
The next result provides such a formula. 
An analog statement holds in any two manifolds related by surgeries, but we prefer to state in the case which we are interested in.
Denote by $\Dpqr$ the number~$pqr-pq-qr-pr=pqr(1-\frac1 p-\frac 1 q-\frac 1 r)$. 
It is the order of the group~$\H_1(\U\Spqr; \ZZ)$. 
Denote by~$\Qpqr$ the bilinear form on~$\RR^3$ described by the matrix~
$\left(\begin{smallmatrix}
qr-q-r & r & q\\
r & pr-p-r & p\\
q & p & pq-p-q 
\end{smallmatrix}\right)$.

\begin{lemma}[Evolution of linking numbers]
	\label{L:Linking}
	For $L_1, L_2$ two disjoint links in~$\Sph^3$ that are also disjoint from the Hopf link~$\Hopf$, their linking number after performing surgeries of respective index~$(p-1, q-1, r-1)$ on the components $(H_1, H_2, H_3)$ of~$\Hopf$ is given by
	\[\lk_{\U\Spqr}(L_1, L_2) = \lk_{\Sph^3}(L_1, L_2) + \frac 1 \Dpqr \Qpqr
	\left( 
		\left(\begin{smallmatrix}
			\lk_{\Sph^3}(L_1, H_1)\\
			\lk_{\Sph^3}(L_1, H_2)\\
			\lk_{\Sph^3}(L_1, H_3)
		\end{smallmatrix}\right),
		\left(\begin{smallmatrix}
			\lk_{\Sph^3}(L_2, H_1)\\
			\lk_{\Sph^3}(L_2, H_2)\\
			\lk_{\Sph^3}(L_2, H_3)
		\end{smallmatrix}\right)
	\right)
	.\]
\end{lemma}

\begin{proof}
	Let $S_2$ be a simplicial integral 2-chain in~$\Sph^3$ bounded by~$\Dpqr L_2$. 
	After possibly canceling pairs of intersection points with different orientations by tunneling, we can assume that $S_2$ intersects each component~$H_i$ of the Hopf link in~$\vert \lk_{\Sph^3}(L_2, H_i) \vert$ points. 
	Let $\nu(\Hopf)$ be a tubular neighborhood of~$\Hopf$. 
	The boundary of~$S_2$ in~$\Sph^3\setminus\nu(\Hopf)$ is then made of $\lk_{\Sph^3}(L_2, H_i)$ meridian circles on every component~$H_i$.

%	For constructing a 2-chain bounded by~$L_2$ in~$\U\Spqr$, it is then enough to add to~$S_2$ a 2-chain bounded by $(p-1)\lk_{\Sph^3}(L_2, H_1)$ longitudes of~$H_1$ (in~$\U\Spqr$), $(q-1)\lk_{\Sph^3}(L_2, H_2)$ longitudes of~$H_2$, and $(r-1)\lk_{\Sph^3}(L_2, H_3)$ longitudes of~$H_3$. 

	Let~$D_1$ be a punctured disc in~$\Sph^3\setminus \nu(\Hopf)$ bounded by a longitude of~$H_1$, a meridian of~$H_2$ and a meridian of~$H_3$. 
	Define similarly~$D_2$ and $D_3$. 
	For $a_1, a_2, a_3$ in~$\ZZ$, the boundary of the 3-chain $a_1 D_1 + a_2 D_2 + a_3 D_3$ then consists of a cycle of $a_1$ longitudes and $a_2+a_3$ meridians of~$H_1$, $a_2$ longitudes and $a_1+a_3$ meridians of~$H_2$, and $a_3$ longitudes and $a_1+a_2$ meridians of~$H_3$.
	
	The 2-chain $S_2 +a_1 D_1 + a_2 D_2 + a_3 D_3$ can then be completed by adding meridian discs into a 2-chain~$\bar S_2$ in~$\U\Spqr$ 
	whose boundary only consists of~$\Dpqr L_2$ 
	if and only if the restriction of the boundary of $S_2 +a_1 D_1 + a_2 D_2 + a_3 D_3$ to $\bord\nu(\Hopf)$ consists of curves with slope $p-1$ on~$\bord H_1$ (\emph{resp.} $q-1$ on~$\bord H_2$, \emph{resp.} $r-1$ on~$\bord H_3$), 
	so if and only if $a_1, a_2, a_3$ satisfy the system
	\begin{eqnarray*}
		\Dpqr\lk_{\Sph^3}(L_2, H_1) + a_2 + a_3 &=& (p-1) a_1\\
		\Dpqr\lk_{\Sph^3}(L_2, H_2) + a_1 + a_3 &=& (q-1) a_2\\
		\Dpqr\lk_{\Sph^3}(L_2, H_3) + a_1 + a_2 &=& (r-1) a_3.
	\end{eqnarray*}

	Since $ \left(\begin{smallmatrix}
		1-p & 1 & 1\\
		1 & 1-q & 1\\
		1 & 1 & 1-r 
	\end{smallmatrix}\right)^{-1} = 
	\frac{-1} \Dpqr
	\left(\begin{smallmatrix}
		qr-q-r & r & q\\
		r & pr-p-r & p\\
		q & p & pq-p-q 
	\end{smallmatrix}\right)$, the solution of the system is given by $		
	\left(\begin{smallmatrix}
		a_1\\a_2\\a_3
	\end{smallmatrix}\right) = 
%	\frac 1\Dpqr
	\left(\begin{smallmatrix}
		qr-q-r & r & q\\
		r & pr-p-r & p\\
		q & p & pq-p-q 
	\end{smallmatrix}\right)
	\left(\begin{smallmatrix}
		\lk_{\Sph^3}(L_2, H_1)\\
		\lk_{\Sph^3}(L_2, H_2)\\
		\lk_{\Sph^3}(L_2, H_3)
	\end{smallmatrix}\right).$	
	
	Finally, the 1-chain $L_1$ intersects~$\Dpqr\lk_{\Sph^3}(L_1, L_2)$ times~$S_2$, and $\lk_{\Sph^3}(L_1, H_i)$ times the disc~$D_i$ for every~$i$, so its intersection with~$\bar S_2$ is equal to
	\[\Dpqr\lk_{\Sph^3}(L_1, L_2) + \sum_i a_i\lk_{\Sph^3}(L_1, H_i).\]
	Dividing by~$\Dpqr$ yields the desired formula.
%	After surgery, the boundary of~$S_2$ in~$\U\Spqr$ then consists of $\vert \lk_{\Sph^3}(L_2, H_1) \vert$ curves with slope~$p-1$ on~$H_1$, \emph{etc.} 
%	For constructing a 2-chain bounded by~$L_2$ in~$\U\Spqr$, it is then enough to add to~$S_2$ a 2-chain bounded by $(p-1)\lk_{\Sph^3}(L_2, H_1)$ longitudes of~$H_1$ (in~$\U\Spqr$), $(q-1)\lk_{\Sph^3}(L_2, H_2)$ longitudes of~$H_2$, and $(r-1)\lk_{\Sph^3}(L_2, H_3)$ longitudes of~$H_3$. 
%	A possibility for such a 2-chain is to take a combination of the discs~$D_1, D_2, D_3$ bounded by~$H_1, H_2, H_3$ in~$\Sph^3$.	 
\end{proof}

Let us test the above formula on Example~\ref{Ex}: the linking number between any two regular fibers of the unit tangent bundle~$\U\Sigma$ equals $- 1/ {\chi(\Sigma)}$ on any $n$-conic sphere~$\Sigma$, so the linking number equals $\frac{pqr}\Dpqr$ when $\Sigma= \Spqr$.
In our presentation, two fibers of the unit tangent bundle correspond to two fibers of the Hopf fibration, thus have linking number~$+1$ in $\Sph^3$ and $+1$ with every component of the Hopf link. 
By Lemma~\ref{L:Linking} their linking number is equal to $1+\frac1\Dpqr\Qpqr((1,1,1),(1,1,1)) = \frac{pqr}\Dpqr$ in~$\U\Sigma$, as expected.

Lemma~\ref{L:Linking} admits a simpler expression when it is applied to some orbits of the template~$\Tpqr$. 
For $\gamma$ an orbit of~$\Tpqr$, denote by $\nombre a \gamma$ and $\nombre b \gamma$ the respective numbers of letters~$a$ and~$b$ in the code of~$\gamma$. 
Denote by $\Qred$ the bilinear form on~$\RR^2$ given by the matrix 
\[\begin{pmatrix}
qr-q-r & -r\\
-r & pr-p-r 
\end{pmatrix}.\]
%\[\Qred((u,v), (u',v')) := (qr-q-r)uu' - ruv' - rvu' + (pr-p-r)vv'.\] 
For~$\gamma, \gamma'$ two orbits of~$\Tpqr$, denote by~$\cro(\gamma, \gamma')$ their \term{crossing number} on~$\Tpqr$, that is, the number of double points using the standard projection of the template (as in Fig.~\ref{F:Tpqr}).

\begin{lemma}
\label{L:LinkingRed}
	For~$\gamma, \gamma'$ two orbits of~$\Tpqr$, one has
	\begin{equation}\tag{*}
		\label{Eq:lk}
	\lk_{\U\Spqr}(\gamma, \gamma') = -\small{\frac 1 2}\cro(\gamma, \gamma') + {\frac 1 \Dpqr} \Qred\left((\nombre a \gamma,\nombre b \gamma),(\nombre a \gamma',\nombre b \gamma')\right).
	\end{equation}
\end{lemma}

\begin{proof}
Since all crossings on~$\Tpqr$ are negative, one has $\lk_{\Sph^3}(\gamma, \gamma') = -\frac 1 2\cro(\gamma, \gamma')$. 
Also, one checks on Figure~\ref{F:Tpqr} that $\lk_{\Sph^3}(\gamma, H_1) = -\nombre a \gamma, \lk_{\Sph^3}(\gamma, H_2) = \nombre b \gamma$ and $\lk_{\Sph^3}(\gamma, H_3) = 0$. 
Plotting these formulas into Lemma~\ref{L:Linking}, we obtain 
	\begin{eqnarray*}
	\lk_{\U\Spqr}(\gamma, \gamma') &=& -\frac 1 2\cro(\gamma, \gamma') + \frac 1 \Dpqr \Qpqr((-\nombre a \gamma,\nombre b \gamma,0),(-\nombre a \gamma',\nombre b \gamma',0))\\
	&=& -\frac 1 2\cro(\gamma, \gamma') + \frac {qr-q-r} \Dpqr (\nombre a \gamma)(\nombre a \gamma') - \frac {r} \Dpqr (\nombre a \gamma) (\nombre b \gamma')\\
	&& - \frac {r} \Dpqr (\nombre b \gamma) (\nombre a \gamma') + \frac {pr-p-r} \Dpqr (\nombre b \gamma) (\nombre b \gamma')\\
	&=& -\frac 1 2\cro(\gamma, \gamma') + \frac 1 \Dpqr \Qred((\nombre a \gamma,\nombre b \gamma),(\nombre a \gamma',\nombre b \gamma')).
	\end{eqnarray*}
	
\vspace{-4mm}
\end{proof}

%%%%%%%%%%%%%%%%%%%%%%%%%%%%%%%%%%%%%
\section{Main computation}
\label{S:Computation}

In this section we prove the hard part of Theorem~\ref{T:Main},  
namely that the linking number of any two periodic orbits of~$\Tpqr$ is negative. 
In Lemma~\ref{L:LinkingRed}, the term~$-\frac 1 2\cro(\gamma, \gamma')$ contributes with the desired sign to the linking number, whereas the second term contributes positively when, for example, $\gamma$ and $\gamma'$ contains only the letter $a$.
We will see that this contribution is always compensated by the first term. 
However, this compensation only holds for the orbits of~$\Tpqr$, not for two arbitrary words, so we will use in a crucial way the fact that~$\Tpqr$ is a strict subtemplate of the Lorenz template.  
In particular, it is necessary that the orbits are balanced in the sense that the code of an orbit cannot contain only one letter. 

The notion of concatenation of words plays a central role in our proof. 
For $u$, $v$ two finite words, their \term{concatenation} $uv$ is the word obtained by reading first $u$ and then $v$. 
In our context the point is that if an infinite word $w^\infty$ describes a periodic orbit on a Lorenz-like template and if $w$ is long enough, then the finite word $w$ may be decomposed as a concatenation $w=uv$ such that $u^\infty$ and $v^\infty$ also describe two periodic orbits of the same Lorenz-like template (Section~\ref{S:Reduction} below is hopefully self-contained, but one may also consult~\cite[Section 3.1.2]{GHS}). 
This phenomenon can be seen as a consequence of the existence of Markov partitions for Anosov flows~\cite{Ratner}.

We first detail the scheme of the proof in the case $3\le p\le q\le r$ in \S~\ref{S:Proof}, and prove the needed lemmas in the next subsections. 
Finally in \S~\ref{S:p=2} we adapt the proof to the case $p=2$.
For simplicity, we now write~$\lk$ instead of~$\lk_{\U\Spqr}$.

%%%%%%%%%%%%%%%
\subsection{Proof of Theorem~\ref{T:Main} in the case $p,q,r\ge 3$}
\label{S:Proof}

By Lemma~\ref{L:Anosov} it is enough to prove that the linking number of every pair of periodic orbits of the geodesic flow on~$\U\Spqr$ is negative. 
By Theorem~\ref{T:Template} it is then enough to prove that the linking number of every pair~$(\gamma, \gamma')$ of periodic orbits of the template~$\Tpqr$ is negative. 

The strategy is as follows. 
We use Equation~\eqref{Eq:lk} of Lemma~\ref{L:LinkingRed}.
First we show that the expression~$-\frac 1 2\cro(\gamma, \gamma')$ behaves subadditively when concatenating words (Lemma~\ref{L:Cutting}). 
Since the form $\Qred$ behaves addititively under concatenation, by Lemma~\ref{L:LinkingRed}, $\lk(\gamma, \gamma')$ also behaves subadditively. 
For $p,q,r$ fixed we can then restrict our attention to the set of extremal orbits, which are determined in Lemma~\ref{L:Extremal}: extremal orbits are encoded by the words $(a^{p-1}b)^ka^ib^j$ or  $(ab^{q-1})^ka^ib^j$ with $(i,j,k)\in([\![1,p-1]\!]\times[\![1,q-1]\!]\setminus\{(1,q-1),(p-1,1)\})\times[\![0,\frac{r-2}2]\!]$, or $(a^{p-1}b)^k(ab^{q-1})^l$ for $(k,l)$ in~$[\![1,\frac{r-2}2]\!]\times[\![1,\frac{r-2}2]\!]$.
 
We then show that the linking numbers of all pairs of such extremal orbits are negative. 
We cover all possibilities in four separate statements (the most critical case is covered by Lemma~\ref{L:wijk}).

\begin{proof}[]
\begin{tabular}{ r | c c c }
    & $(a^{p-1}b)^ka^ib^j$ & $(ab^{q-1})^ka^ib^j$ & $(a^{p-1}b)^k(ab^{q-1})^l$ \\
   \hline
   $(a^{p-1}b)^ka^ib^j$ & Lemma~\ref{L:wijk} & Lemma~\ref{L:Oppose} &  Lemma~\ref{L:Mixed}  \\
   $(ab^{q-1})^ka^ib^j$ &  & Lemma~\ref{L:wijk} &  Lemma~\ref{L:Mixed}  \\
   $(a^{p-1}b)^k(ab^{q-1})^l$ & & & Lemma~\ref{L:Mixed2} \\
\end{tabular}

\vspace{1mm}
The rest of the section is dedicated to proving these four lemmas.
\end{proof}

\noindent{\bf Remark.}
	For~$p,q,r$ fixed, Lemmas~\ref{L:wijk}, \ref{L:Oppose}, \ref{L:Mixed2} and \ref{L:Mixed} are statements that involve finitely many computations only. 
	On the one hand, the proofs we propose here are rather heavy and we are not fully satisfied with them. 
	On the other hand, the crossing number of a pair of orbits of a Lorenz-type template is easy to compute, therefore, using Equation~\eqref{Eq:lk}, the inequalities in the lemmas are easily checked. 
	We did so into a Sage\footnote{\url{http://www.sagemath.org}} worksheet available on our website\footnote{\url{	https://www-fourier.ujf-grenoble.fr/~dehornop/maths/ComputationsLinkingTpqr.sws}}. 	
	It took about 25 seconds on a laptop to check the validity of these lemmas for all $p\le 4, q\le5, r\le7$ and about 2 hours for all $p\le 6, q\le8, r\le10$.
	
%%%%%%%%%%%%%%%%%
\subsection{Reducing to a finite number of estimations}
\label{S:Reduction}

%For $v,w$ two finite words, we denote by~$\cros(v^\infty,w^\infty)$ the crossing number of the two orbits of the Lorenz template that are coded by~$v^\infty$ and $w^\infty$. 
The next two results of this section imply that for proving that any two orbits of~$\Tpqr$ have negative linking number ($p,q,r$ being fixed), it is enough to restrict our attention to an explicit finite list of orbits.

\begin{definition}
	A \term{cut} of a finite word~$w$ in the alphabet~$\{a, b\}$ is a pair of words~$u,v$, called the \term{factors}, such that 
	\begin{itemize}
	\item $w$ is obtained by cyclic permutation of the letters of~$uv$, 
	\item $u$ and $v$ end with the letters $a$ and $b$ respectively, 
	\item $u$ and $v$ satisfy $u^\infty < v^\infty$, 
	\item no shift of $u^\infty$ or $v^\infty$ lies between $u^\infty$ and $v^\infty$ in the lexicographic order.
	\end{itemize}
	A cut is \term{admissible} if the two factors code orbits of~$\Tpqr$, that is, if all their shifts are between the kneading sequences given by Table~\ref{Table}.
\end{definition}

Graphically, if $w^\infty$ is the code of an orbit~$\gamma$, a cut correspond suppressing from $\gamma$ a corner of the bottom border of the diagram of~$\gamma$ (see Fig.~\ref{F:Cut}). 

\begin{figure}[ht]
	\label{D:Cut}
	\includegraphics*[width=.6\textwidth]{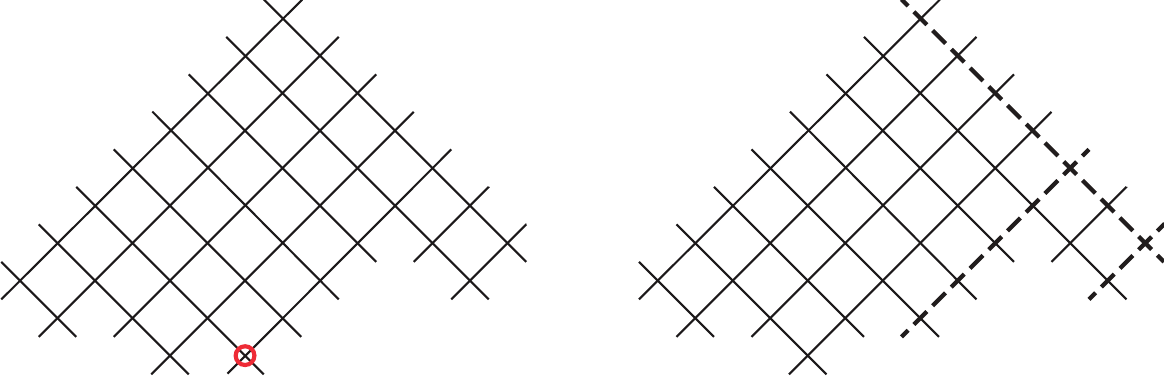}
 	\caption{\small A cut of the word $aa\vert abb\vert abbababbab$ (on the left) creates the two new words $abbababbabaa$ and $abb$ (on the right).}
  	\label{F:Cut}
\end{figure}

It may not be obvious how a cut could \term{not} be admissible. 
The point is that cutting an orbit creates two orbits, one which corresponds to a part of the original orbit and is moved slightly to the left, and the other one which is moved to the right. 
This moving could make one of the two new orbits exit the template. 

By extension, for $v,w$ two finite words, we denote by~$\cros(v^\infty,w^\infty)$ the crossing number of the two orbits of the Lorenz template that are coded by~$v^\infty$ and $w^\infty$. 
The key-property of cuts is

\begin{lemma}[Superadditivity of the crossing number]
	\label{L:Cutting}
	Assume that $u, v$ are the factors of a cut. 
	Then for every finite word~$x$, we have
	\[ \cros((uv)^\infty, x^\infty) \ge \cros(u^\infty, x^\infty) + \cros(v^\infty, x^\infty).\]
\end{lemma}

\begin{proof}
	Call an \term{(elementary) arc} a piece of orbit of the template delimited by two (consecutive) intersections points with the branching arc of the template. 
	An elementary arc then corresponds to a letter of its code. 
	The orbit coded by~$(uv)^\infty$ is parallel and situated on the left of~$u^\infty$ for the first $\vert u\vert$ elementary arcs 
	and then parallel and to the right of~$v^\infty$ for the next~$\vert v\vert$ elementary arcs. 
	Thus the correspondence between the letters of~$uv$ and the letters of $u$ and $v$ induces a canonical correspondence between elementary arcs of $(uv)^\infty$ and elementary arcs of $u^\infty$ or of $v^\infty$. 
	Call the \term{intermediate zone} the zone situated between the elementary arcs of $(uv)^\infty$ 
	and the corresponding arcs of $u^\infty$ or of $v^\infty$ (in blue and green on Fig.~\ref{F:EmbrassedCutting}).

	First suppose that no arc of~$x^\infty$ travels in the intermediate zone. 
	Then an elementary arc of $x^\infty$ intersects an elementary arc of $(uv)^\infty$ 
	if and only if it intersects the corresponding elementary arc of $u^\infty$ or of $v^\infty$. 
	Therefore we have $\cros((uv)^\infty, x^\infty) = \cros(u^\infty, x^\infty) + \cros(v^\infty, x^\infty)$ in this case.

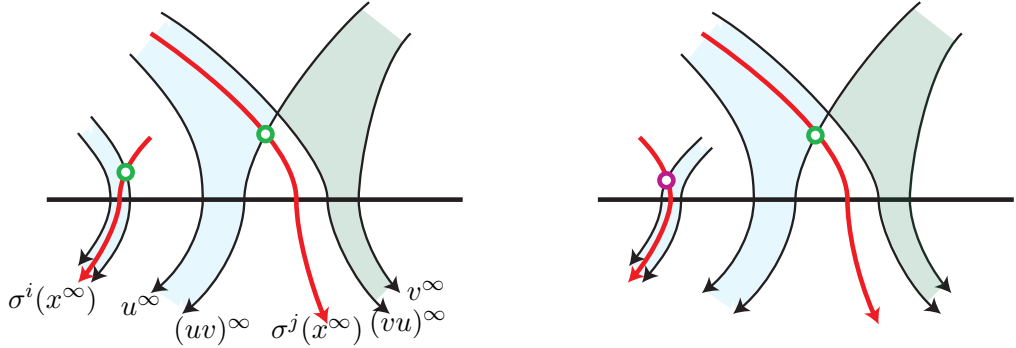
\begin{figure}[ht]
   	\begin{picture}(130,43)(0,0)
      	\put(0,0){\includegraphics*[width=.9\textwidth]{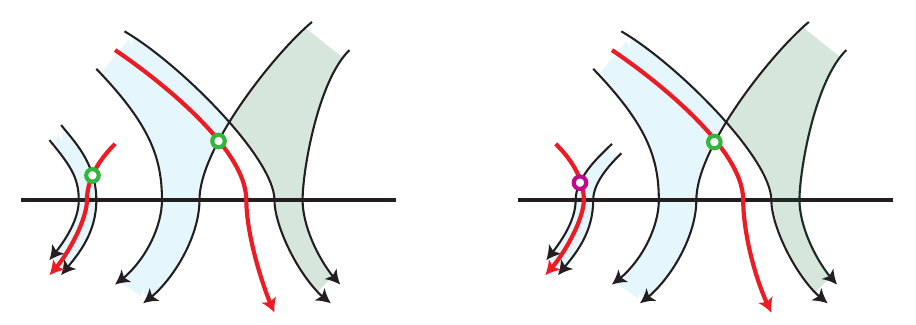}}
      	\put(13,3){$u^\infty$}
      	\put(20,0){$(uv)^\infty$}
      	\put(46,1){$(vu)^\infty$}
      	\put(51,5){$v^\infty$}
      	\put(-2,4){$\sigma^i(x^\infty)$}
      	\put(33,0){$\sigma^j(x^\infty)$}
   	\end{picture}
   	\caption{\small Proof of the superadditivity lemma: when an arc of $x^\infty$ (in red) visits 
	the region between $u^\infty$ and $(uv)^\infty$ (in blue), either it crosses $(uv)^\infty$ twice 
	and does not cross $u^\infty$ (on the left), or it crosses both orbits once (on the right).}
  	\label{F:EmbrassedCutting}
\end{figure}

	In the general case, an arc of~$x^\infty$ can only enter the intermediate zone by intersecting an arc of $(uv)^\infty$. 
	If it exists the zone by intersecting an arc of $u^\infty$ or~$v^\infty$, then these two points are canonically associated. 
	Otherwise it exits by cutting another arc of $(uv)^\infty$ 
	and then we get two more intersection points with $(uv)^\infty$ than with the union of $u^\infty$ and $v^\infty$. 
	In both cases, the inequality holds.
\end{proof}

%\begin{definition}
%	\label{D:Decomposable}
%	A word $w$ is \term{decomposable} if it has an admissible cut.
%\end{definition}

Call a periodic orbit of the template~$\Tpqr$ \term{extremal} if it has no admissible cut. 
Such extremal orbits correspond to the \emph{elementary loops} of Fried~\cite{FriedCross}.

\begin{lemma}
	\label{L:Extremal}
	A periodic orbit of~$\Tpqr$ is extremal if and only if its code is of the form $a^ib^j(a^{p-1}b)^k$ or  $a^ib^j(ab^{q-1})^k$ with $(i,j,k)\in([\![1,p-1]\!]\times[\![1,q-1]\!]\setminus\{(1,q-1),(p-1,1)\})\times[\![0,\frac{r-2}2]\!]$, or $(a^{p-1}b)^k(ab^{q-1})^l$ for $(k,l)$ in~$[\![1,\frac{r-2}2]\!]\times[\![1,\frac{r-2}2]\!]$.
\end{lemma}

\begin{proof}
	Let $\gamma$ be an extremal orbit and denote by~$w$ its code.
	By Theorem~\ref{T:Template} and Table~\ref{Table}, 
	the word $w$ does not contain more than~$p-1$ consecutive~$a$ and $q-1$ consecutive~$b$. 
	Also it does not contain more than~$\frac{r-2}2$ consecutive blocks of the form $a^{p-1}b$ or $ab^{q-1}$.

	First suppose that~$w$ contains no syllable of the form~$a^{p-1}b$ or $ab^{q-1}$.	
	Decompose the axis of the template into the union of segments, corresponding to orbits starting with 
%	$(a^{p-1}b)^{\lfloor r/2\rfloor}$, $(a^{p-1}b)^{\lfloor r/2\rfloor-1}, \dots a^{p-1}b$, 
	$a^{p-1}b$, $a^{p-2}b, \dots, ab$, 
	$ba, b^2a, \dots, b^{q-2}a$, $b^{q-1}a$ (see Fig.~\ref{F:Division}).
%	$(b^{q-1}a)^2, \dots, (b^{q-1}a)^{\lfloor r/2\rfloor}$. 
\begin{figure}[ht]
   	\begin{picture}(80,50)(0,0)
      	\put(0,0){\includegraphics*[width=.5\textwidth]{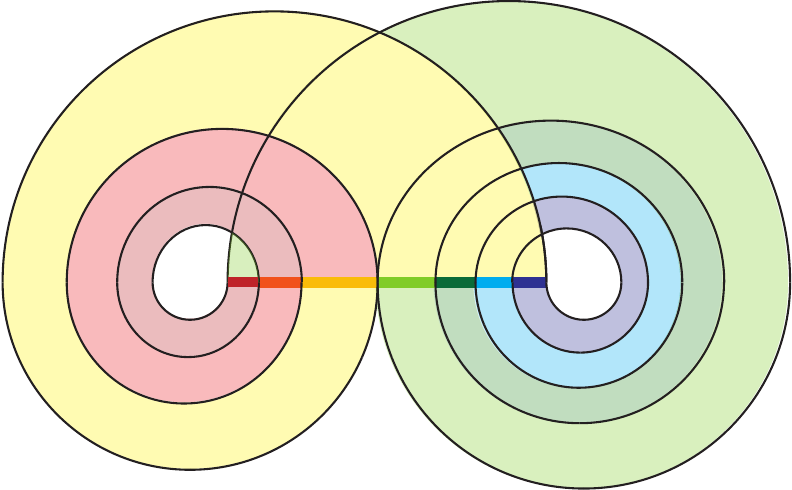}}
   	\end{picture}
   	\caption{\small Division of the axis of~$\Tpqr$ into segments corresponding to orbits starting with $a^3b, a^2b, ab, ba, b^2a, b^3a$, and $b^4a$ here.}
  	\label{F:Division}
\end{figure}
	Now travel along~$\gamma$ starting from the leftmost point  and consider the points where it intersects the axis of the template. 
%	For convenience, when~$\gamma$ travels along a block~$a^{p-1}b$ or $b^{q-1}a$, we do not consider the intermediate intersection points.
	By the hypothesis that~$\gamma$ is the leftmost orbit in the corresponding subsegment, then either $\gamma$ is always the leftmost orbit of each of the visited subsegments, in which case $\gamma$ comes back to its initial point after having visited at most once every subsegment, and in this case~$w$ is of the form~$a^ib^j$. 
	Or at some point~$\gamma$ stops being the leftmost orbit in the corresponding subsegment, which means that some arc comes from the right to pass over~$\gamma$ and sits just to the left of~$\gamma$ in the corresponding subsegment. 
	This determines a place to cut~$\gamma$. 
	Since the two arcs that intersect at that cut are in the same subinterval, they correspond to letters of~$w$ that are followed by the same number of letters of the same type. 
	Therefore cutting at this place amounts to factor $w=a^{i_1}b^{j_1}\dots a^{i_n}b^{j_n}$ into $a^{i_k}b^{j_k}\dots a^{i_l}b^{j_l}$ and $a^{i_{l+1}}b^{j_{l+1}}\dots a^{i_{k-1}}b^{j_{k-1}}$. 
	These two words still encode orbits of the template, so the cut is admissible.
		
	Now allow~$w$ to contain syllables of the form~$a^{p-1}b$ or~$ab^{q-1}$. 
	We apply the same strategy for finding an admissible cut works, except that we have to forget about the blocks of the form~$(a^{p-1}b)^k$ or~$(ab^{q-1})^k$. 
	Namely, we follow the orbit~$\gamma$ encoded by~$w$ starting from its leftmost point, but we consider only the relative place of~$\gamma$ in the corresponding subsegments when we are visiting syllables of the form~$(a^{p-1}b)^k$ or~$(ab^{q-1})^k$. 
	One still finds a cut. 
	By our assumption, this cut takes place between two letters of~$\gamma$ not belonging to a syllable of the form~$(a^{p-1}b)^k$ or~$(ab^{q-1})^k$.		
\end{proof}

%%%%%%%%%%%%%%%%%%%%%
\subsection{Estimating $\lk(a^ib^j, a^{i'}b^{j'})$}

Since any word can be decomposed as a product of syllables of the form~$a^ib^j$, using the subbativity of linking numbers (Lemma~\ref{L:Cutting}), it is enough for proving Theorem~\ref{T:Main} to show that $\lk(a^ib^j, a^{i'}b^{j'})$ is negative for all possible values of $(i,j,i',j')$. 
Unfortunately, this is not the case: using Lemma~\ref{L:LinkingRed}, one sees that $\Dpqr\lk (a^{p-1}b, a^ib^j) = qi - jp$ and in particular we have $\lk(a^{p-1}b, a^{p-1}b) =(pq-p-q)/\Dpqr>0$. 
This formula is not a surprise since the code $a^{p-1}b$ actually corresponds to a $+2\pi/r$-rotation around the order~$r$-point of~$\Spqr$, hence the curve with code $(a^{p-1}b)^r$ is isotopic to a fiber, and we know that the linking number between two fibers is positive (by Example~\ref{Ex}).
However, the next two lemmas state that this situation is almost the only bad case.

\begin{lemma}[easy case]
	\label{L:EasyCase}
	If $i<i', j< j'$, then $\lk(a^ib^j, a^{i'}b^{j'})$ is negative.
\end{lemma}

\begin{proof}
	Figure~\ref{F:Crossing} shows that we have $\cro(a^ib^j, a^{i'}b^{j'}) =2 (i+j)$ in this case.

\begin{figure}[ht]
     	\includegraphics*[width=.6\textwidth]{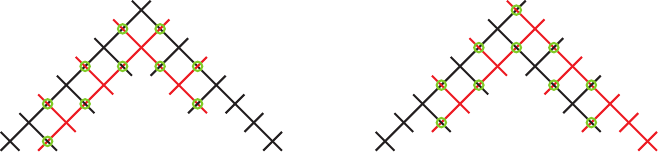}
  	\caption{\small Crossing number between two orbits of codes $a^ib^j$ (red) and $a^{i}b^{j'}$ (black), with $i<i', j<j'$ on the left and $i\le i', j\ge j'$ on the right. The double points are encircled, they are $2(i+j)$ on the left and $2(i+j'-1)$ on the right.}
  	\label{F:Crossing}
\end{figure}

	By Lemma~\ref{L:LinkingRed}, we have $\Dpqr\lk(a^ib^j, a^{i'}b^{j'}) = -iqr (1-\frac 1 q - \frac 1 r - \frac j {iq})(p-i') - jpr (1-\frac 1 p - \frac 1 r - \frac i {jq})(q-j')$.

	Suppose $i, j\ge 2$ and $(i,q,r), (j,p,r)\neq (2,3,4), (2,3,5)$.
	Since $j<q$, we have $\frac j {iq} < \frac 1 i$, and therefore $1-\frac 1 q - \frac 1 r - \frac j {iq} > 1-\frac 1 q - \frac 1 r - \frac 1{i} \ge 0$. 
	Since $i'<p$, the term $-iqr (1-\frac 1 q - \frac 1 r - \frac j {iq})(p-i')$ is always negative. 
	Similarly for the term $- jpr (1-\frac 1 p - \frac 1 r - \frac i {jq})(q-j')$. 
	Then $\lk(a^ib^j, a^{i'}b^{j'})$ is the sum of two negative terms, hence is negative 

	If $i=1$, since $j\le q-2$, we have $1-\frac 1 q - \frac 1 r - \frac j {iq} = 1-\frac {j+1} q - \frac 1 r \ge 1-\frac {q-1}q - \frac 1 r = \frac 1 q - \frac 1 r$. 
	Using the assumption $q\le r$, we still have a non-positive term. 
	The same holds for $j=1$. 
	Note that the equality holds only for $j=q-2$ ({\it resp.} $i=p-2$), 
	therefore the sum can be equal to~$0$ only if $i=j=1=p-2=q-2$ and $p=q=r$, which forces $p=q=r=3$. 
	These values do not correspond to a hyperbolic orbifold. 

	In the case $(i,q,r) = (2,3,4)$ or $(2,3,5)$, since $j<j'<q$, we necessarily have $j=1, j'=2$. 
	In this case, an easy computation leads to $\Dpqr\lk(a^ib^j, a^{i'}b^{j'}) = 6i'-9p+12$. Since $i'\le p-1$, we have $6i'-9p+12 \le -3p+6 <0$. 
	The case $(j,p,r) = (2,3,4)$ or $(2,3,5)$ is treated in the same way.

	Therefore, for $i<i', j<j'$, we always have $\lk(a^ib^j, a^{i'}b^{j'}) < 0$.
\end{proof}

\begin{lemma}[hard case]
	\label{L:HardCase}
	If $i\le i', j\ge j'$, then $\lk(a^ib^j, a^{i'}b^{j'})$ is negative, except if $(i,j)=(1,q-1)$ or $(i',j')=(p-1,1)$. 
\end{lemma}

\begin{proof}
	Looking at Figure~\ref{F:Crossing}, we now have $\cro(a^ib^j, a^{i'}b^{j'}) =2( i+j'-1)$.

	Then $\Dpqr\lk(a^ib^j, a^{i'}b^{j'}) = -[(qr-q-r)i-rj'](p-i') - [(pr-p-r)j'-ri](q-j) - r(i'-i)(j-j') + (pqr-pq-pr-qr)$. The expression is linear in the four variables $i,j,i',j'$. 

	In order to prove the lemma, we have to prove that on the extremal points of the region 
	$D_{i,j,i',j'}:=\{(i,j,i',j')\vert i\le i', j\ge j'\}\setminus \{(1,q-1,*,*)\}\cup\{(*,*,p-1,1)\}$, 
	the number $\Dpqr\lk(a^ib^j, a^{i'}b^{j'})$ is negative (a point of~$D_{i,j,i',j'}$ being \emph{not} extremal if it lies between two points of the domain that share 3 coordinates).

	By linearity, for each value of $(i,j)$, the extremum is reached when $(i',j')$ is one of the five points~$(i,j), (i,1), (p-2,1), (p-1,2), (p-1,j)$. 
	By symmetry, we can restrict ourselves to the values $(i,j), (i,1)$, and~$(p-2,1)$.

	\begin{itemize}
	\item
	\emph{Case 1: $(i',j')=(i,j)$.} $\Dpqr\lk(a^ib^j, a^{i'}b^{j'})$ is then equal to $$(qr-q-r)i^2 + (pr-p-r)j^2 -2rij -\Dpqr(i+j-1).$$
	Since the coefficients of $i^2$ and $j^2$ are positive, the maximum is reached on the boundary of the domain, 
	hence we only have to check the points $(i,j) = (1,1), (p-2,1), (p-1,2), (p-1,q-1), (2, q-1), (1, q-2)$. By symmetry, we can actually only consider $(1,1)$ and $(p-2,1)$. 

		\begin{itemize}
		\item
		\noindent\emph{Case 1.1: $(i,j,i',j')=(1,1,1,1)$}.
		We have $\Dpqr\lk(a^1b^1, a^{1}b^{1}) = -pqr +pq+2pr+2qr-p-q-4r = -(p-2)(q-2)(r-2) - (p-3)(q-3) +1$. 
		The first term is always smaller than $-1$, the second is non-positive, so the sum is negative. 

		\item
		\noindent\emph{Case 1.2: $(i,j,i',j')=(p-2,1,p-2,1)$}.
		We have $\Dpqr\lk(a^{p-2}b^1, a^{p-2}b^{1}) = -pqr +2pq+pr+2qr-p-4q-r = -(p-2)(q-2)(r-2) - (p-3)(r-3) +1$, which is negative for the same reason as in Case 1.1.

		\end{itemize}

	\item
	\emph{Case 2: $(i',j')=(i,1)$.} $\Dpqr\lk(a^ib^j, a^{i'}b^{1})$ is then equal to $$(pr-p-r)i^2 -\Dpqr i +(pr-p-r)j-ri(j+1).$$
	Since the coefficient of $i^2$ is positive, the maximum is reached on the boundary of the domain, 
	hence we only have to check the points $(i,j) = (1,1), (p-2,1), (p-2,q-1), (2, q-1), (1, q-2)$. (There is no more symmetry.)

		\begin{itemize}
		\item
		\noindent\emph{Case 2.1: $(i,j,i',j')=(1,1,1,1)$}. This case is similar to Case 1.1.

		\item
		\noindent\emph{Case 2.2: $(i,j,i',j')=(p-2,1,p-2,1)$}. This case is similar to Case 1.2.

		\item
		\noindent\emph{Case 2.3: $(i,j,i',j')=(p-2,q-1,p-2,1)$}. 
		We have $\Dpqr\lk(a^{p-2}b^{q-1}, a^{p-2}b^{1}) = -pqr +pq+pr+3qr+p-4q-r = -(p-3)(q-1)(r-2) - (p-2)(q-3)$, which is the sum of two non-positive terms.

		\item
		\noindent\emph{Case 2.4: $(i,j,i',j')=(2,q-1,2,1)$}. 
		We have $\Dpqr\lk(a^{2}b^{q-1}, a^{2}b^{1}) = -pqr +pq+pr+3qr+p-4q-r = -(p-3)(q-1)(r-2) - (p-2)(q-3)$, which is the sum of two non-positive terms.

		\item
		\noindent\emph{Case 2.5: $(i,j,i',j')=(1,q-1,1,1)$}. We have $\Dpqr\lk(a^{1}b^{q-1}, a^{1}b^{1}) = pr+2p-q+2r = -(p-2)(r-2)-q+4$, which is negative.

		\end{itemize}

	\item

	\emph{Case 3: $(i',j')=(p-2,1)$.} $\Dpqr\lk(a^ib^j, a^{i'}b^{1})$ is then equal to $$-(qr-2q-r)i+(r-p)j.$$
	By linearity, we only have to check the points $(i,j) = (1,1), (p-2,1), (p-1,2)$, $(p-1,q-1), (2, q-1), (1, q-2)$. 
		\begin{itemize}
		\item
		\noindent\emph{Case 3.1: $(i,j,i',j')=(1,1,p-2,1)$}. This case is symmetric to Case 2.5.

		\item
		\noindent\emph{Case 3.2: $(i,j,i',j')=(1,q-2,p-2,1)$}. 
		We have $\Dpqr\lk(a^{1}b^{q-2}, a^{p-2}b^{1}) = -pq+2p+2q-r = -(p-2)(q-2) -r+4$, which is negative.

		\item
		\noindent\emph{Case 3.3: $(i,j,i',j')=(2,q-1,p-2,1)$}. 
		We have $\Dpqr\lk(a^{2}b^{q-1}, a^{p-2}b^{1}) = -pq-qr +p+4q+r = -(q-1)(p+r-4) +4$, which is negative.

		\item
		\noindent\emph{Case 3.4: $(i,j,i',j')=(p-1,q-1,p-2,1)$}. 
		We have $\Dpqr\lk(a^{p-1}b^{q-1}, a^{p-2}b^{1}) = -pqr +pq+pr+2qr+p-2q-2r = -(p-2)(q-1)(r-1) +4$, which is negative.

		\item
		\noindent\emph{Case 3.5: $(i,j,i',j')=(p-1,2,p-2,1)$}. 
		We have $\Dpqr\lk(a^{p-1}b^{2}, a^{p-2}b^{1}) = -pqr +2pq+pr+qr-2p-2q+r = -(p-1)(q-1)(r-2) +4$, which is negative.

		\item
		\noindent\emph{Case 3.6: $(i,j,i',j')=(1,q-1,1,1)$}. This case is the same as Case 2.5.
		\end{itemize}
	\end{itemize}
\vspace{-5mm}
\end{proof}

%%%%%%%%%%%%%%%%%%
\subsection{Estimating $\lk((a^{p-1}b)^ka^ib^j, (a^{p-1}b)^{k'}a^{i'}b^{j'})$}

As we have seen, syllables of the form $a^{p-1}b$ may contribute positively to the linking number. However, using the information we have on the admissible codes, we know that there cannot be more than $\frac{r-2}2$ consecutive such syllables. 
By the subadditivity of the linking number and using the expression $\Dpqr\lk (a^{p-1}b, a^ib^j) = qi - jp$, we have the inequality 
	\begin{eqnarray*}
	\Dpqr\lk((a^{p-1}b)^ka^ib^j, (a^{p-1}b)^{k'}a^{i'}b^{j'})
	&\le& \Dpqr kk' \lk(a^{p-1}b,a^{p-1}b) +\Dpqr k \lk(a^{p-1}b, a^{i'}b^{j'}) \\
	&&+ \Dpqr k' \lk(a^{p-1}b, a^ib^j) + \Dpqr \lk(a^ib^j, a^{i'}b^{j'})\\
	&=& kk'(pq-p-q) + k(qi'-pj')+k'(qi-jp)\\
	&&+\Dpqr\lk(a^ib^j, a^{i'}b^{j'}).
	\end{eqnarray*}
Unfortunately, the term $\Dpqr \lk(a^ib^j, a^{i'}b^{j'})$ is not always negative enough and cannot always compensate the three first terms. 
So we need a better control than the superadditivity lemma. 
The additional term $2\min(k,k')$ in the next lemma will make the difference.
By convention, for $\gamma$ a periodic orbit of~$\Tpqr$, we denote by~$\cro(\gamma, \gamma)$ twice the number of double points of~$\gamma$ (this is the crossing of~$\gamma$ with a copy of itself slightly translated along the template).

\begin{lemma}
	\label{L:Lkijk}
	The crossing number $\cro((a^{p-1}b)^ka^ib^j, (a^{p-1}b)^{k'}a^{i'}b^{j'})$ is at least equal to 
	$$kk'\cro(a^{p-1}b, a^{p-1}b) + k \cro(a^{p-1}b, a^{i'}b^{j'}) + k' \cro(a^{p-1}b, a^{i}b^{j}) + \cro(a^ib^j, a^{i'}b^{j'}) +2\min(k,k').$$
\end{lemma}

\begin{proof}
	Without loss of generality suppose $k\le k'$ (if $k=k'$, also suppose $i\le i'$, and if moreover $i=i'$ suppose $j\ge j'$). 
	We make an induction on $k$.
	For $k=0$, the result is a particular case of Lemma~\ref{L:Cutting}. 
	
	Now, consider the cut $(a^{p-1}b)^{k}a^{i'}b^{j'} = a^{p-1}\left| b(a^{p-1}b)^{k-1}a^{i'}b^{j'-1}\right| b$, whose factors are then $a^{p-1}b$ and $(a^{p-1}b)^{k-1}a^{i}b^{j}$ (after cyclic permutation). 
	Note that this cut is not admissible since $a^{p-1}b$ does not lie on~$\Tpqr$.
	However the statement only deals with crossing numbers on any Lorenz-type template, so that we can work on a Lorenz template with trivial kneading sequences $a^\infty$ and~$b^\infty$. 
	
	In the proof of Lemma~\ref{L:Cutting} we introduced the intermediate zone and showed that the crossing number $\cro((a^{p-1}b)^{k}a^{i}b^{j}), (a^{p-1}b)^{k'}a^{i'}b^{j'})$ exceeds the sum $\cro(a^{p-1}b, (a^{p-1}b)^{k'}a^{i'}b^{j'}))$ $ + \cro((a^{p-1}b)^{k-1}a^ib^j, (a^{p-1}b)^{k'}a^{i'}b^{j'}))$ by twice the number of arcs of the orbit~$(a^{p-1}b)^{k'}a^{i'}b^{j'}$ that enter the intermediate zone by crossing~$(a^{p-1}b)^{k}a^{i}b^{j}$ (see Fig.~\ref{F:EmbrassedCutting} left). 
	By the assumption $k\le k'$, the leftmost point of $(a^{p-1}b)^{k'}a^{i'}b^{j'}$ is to the left of $(a^{p-1}b)^{k}a^{i}b^{j}$, so that there is one arc of $(a^{p-1}b)^{k'}a^{i'}b^{j'}$ that enters the left part of the intermediate zone by crossing~$(a^{p-1}b)^{k}a^{i}b^{j}$. 
	Also, by $k\le k'$, there is an arc of~$(a^{p-1}b)^{k'}a^{i'}b^{j'}$ situated between $b(a^{p-1}b)^{k-1}a^{i'}b^{j'-1}$ and $b(a^{p-1}b)^{k-1}a^{i'}b^{j'}a^{p-1}$, so that there is also one arc of $(a^{p-1}b)^{k'}a^{i'}b^{j'}$ that enters the right part of the intermediate zone by crossing~$(a^{p-1}b)^{k}a^{i}b^{j}$.  
	Therefore the number $\cro((a^{p-1}b)^{k}a^{i}b^{j}), (a^{p-1}b)^{k'}a^{i'}b^{j'})$ exceeds $\cro(a^{p-1}b, (a^{p-1}b)^{k'}a^{i'}b^{j'})) +$ $\cro((a^{p-1}b)^{k-1}a^ib^j, (a^{p-1}b)^{k'}a^{i'}b^{j'}))$ by at least two. 
	This concludes the induction step.
\end{proof}

\begin{lemma}
\label{L:wijk}
	For $0\le k,k'\le \frac{r-2}2$, the number $\lk((a^{p-1}b)^ka^ib^j, (a^{p-1}b)^{k'}a^{i'}b^{j'})$ is negative.
\end{lemma}

\begin{proof}
	Without loss of generality, we assume $k\le k'$. 
	Thanks to Lemma~\ref{L:Lkijk}, we have 
		\begin{eqnarray*}
		\Dpqr\lk((a^{p-1}b)^ka^ib^j, (a^{p-1}b)^{k'}a^{i'}b^{j'})
		&=& kk'(pq-p-q) + k(qi'-pj')+k'(qi-jp)\\
		&&+\Dpqr \lk(a^ib^j, a^{i'}b^{j'}) +k(-pqr + pq+pr+qr).
		\end{eqnarray*}

	This expression is linear in~$k, k'$. 
	Therefore it is enough to evaluate it for $(k,k') = (0,0), (0,\frac{r-2}2),$ and $(\frac{r-2}2,\frac{r-2}2)$. 

	\begin{itemize}
	\item
	\emph{Case 1: $(k,k') = (0,0)$.}	This corresponds to Lemmas~\ref{L:EasyCase} and \ref{L:HardCase}.
	
	\item
	\emph{Case 2: $(k,k') =(0,\frac{r-2}2)$.} 
%		\begin{eqnarray*}
%		&&\Dpqr\lk((a^{p-1}b)^ka^ib^j, (a^{p-1}b)^{\frac{r-2}2}a^{i'}b^{j'})\\
%		&=& k(pq-p-q)(r-2)/2 + k(qi'-pj')+(qi-pj)(r-2)/2\\
%		&&+\Dpqr \lk(a^ib^j, a^{i'}b^{j'}) +k(-pqr + pq+pr+qr)\\
%		&=& (-pqr+pr+qr+2p+2q+2qi'-2pj')k/2 + (iq-jp)(r-2)/2 +\Dpqr \lk(a^ib^j, a^{i'}b^{j'})
%		\end{eqnarray*}
%	From $i'\le p-1$, we get that the coefficient of $k$ is smaller than $-pqr+pq+qr+pr+2p$, which is negative.
%	So it is enough to evaluate it at $k=0$.
%	
%	In the case $k=0$, one also checks that the coefficient of $k'$ is positive, so both cases reduce to showing that $\Dpqr\lk(a^ib^j, (a^{p-1}b)^{\frac{r-2}2}a^{i'}b^{j'})$ is negative.
%	
	We have
	\[ \Dpqr\lk(a^ib^j, (a^{p-1}b)^{\frac{r-2}2}a^{i'}b^{j'}) 
	= (iq-jp)(r-2)/2 +\Dpqr \lk(a^ib^j, a^{i'}b^{j'}).
	\]
	By subadditivity, the maximum is reached when $(i,j)$ and $(i',j')$ are extremal, that is, equal to $(1,1), (p-2,1), (p-1,2), (p-1,q-1), (2, q-1),$ or $(1, q-1)$. 
	In all cases except $(p-2,1)$ and $(p-1,2)$ the term $(iq-jp)(r-2)/2$ is negative or clearly does not compensate the contribution of $\Dpqr \lk(a^ib^j, a^{i'}b^{j'})$.
	Therefore, the maximum of $\lk(a^ib^j, (a^{p-1}b)^{\frac{r-2}2}a^{i'}b^{j'})$, if positive, is reached for $(i,j,i',j')= (p-1,2, p-1,2), (p-1, 2, p-2, 1), (p-2, 1, p-1, 2), $ or $(p-2,1,p-2,1)$. 
	\begin{itemize}
		\item\noindent\emph{Case 2.1: $(i,j,i',j')=(p-1,2, p-1,2)$}. Then $\Dpqr\lk(a^{p-1}b^2, (a^{p-1}b)^{\frac{r-2}2}a^{p-1}b^{2}) = (pq-2p-q)(r-2)/2 -pqr+2pq+2pr+qr-4p-q-r = -((p-1)(q-2)(r-2)-2(q-2))/2$, which is negative.
		\item\noindent\emph{Case 2.2: $(i,j,i',j')=(p-1,2, p-2,1)$}. Then $\Dpqr\lk(a^{p-1}b^2, (a^{p-1}b)^{\frac{r-2}2}a^{p-2}b^{1}) =  (pq-2p-q)(r-2)/2 -pqr +2pq+pr+qr-2p-2q+r = -((p-1)q(r-2) -2r)/2 $, which is negative.
		\item\noindent\emph{Case 2.3: $(i,j,i',j')=(p-2,1, p-1,2)$}. Then $\Dpqr\lk(a^{p-2}b^1, (a^{p-1}b)^{\frac{r-2}2}a^{p-1}b^{2}) = (pq-p-2q)(r-2)/2 -pqr+2pq+pr+qr-2p-2q+r= -(p(q-1)(r-2)-2r)/2$, which is negative.
		\item\noindent\emph{Case 2.4: $(i,j,i',j')=(p-2,1, p-2,1)$}. Then $\Dpqr\lk(a^{p-2}b^1, (a^{p-1}b)^{\frac{r-2}2}a^{p-2}b^{1}) = (pq-p-2q)(r-2)/2 -pqr +2pq+pr+2qr-p-4q-r = -((p-2)q(r-2)+p+r)/2$, which is negative.

	\end{itemize}
	
	\item
	\emph{Case 3: $(k,k') =(\frac{r-2}2,\frac{r-2}2)$.} 
	We have
		\begin{eqnarray*}
		&&\Dpqr\lk((a^{p-1}b)^{\frac{r-2}2}a^ib^j, (a^{p-1}b)^{\frac{r-2}2}a^{i'}b^{j'})\\
		&=&(-pqr+pr+qr+2p+2q+2q(i+i')-2p(j+j'))(r-2)/2 \\&&+\Dpqr \lk(a^ib^j, a^{i'}b^{j'})
		\end{eqnarray*}
	Since $i,i'\le p-1, j,j'\ge 1$, we have $2q(i+i')-2p(j+j') \le 2(pq-p-q)$. 
	Actually, since both $(i,j)$ and $(i',j')$ cannot be $(p-1,1)$, we even have $2q(i+i')-2p(j+j') \le \max(4(pq-2p-q), 4(pq-p-2q))$.
	%For the final computation we break the symmetry by assuming $p\le q$, so that  $2q(i+i')-2p(j+j') \le 2(pq-2p-q)$. 
	Since $\Dpqr \lk(a^ib^j, a^{i'}b^{j'})$ is always negative (Lemmas~\ref{L:EasyCase} and~\ref{L:HardCase}), we have
		\begin{eqnarray*}
		&&\Dpqr\lk((a^{p-1}b)^{r/2}a^ib^j, (a^{p-1}b)^{r/2}a^{i'}b^{j'})\\
		&< & (-pqr+4pq+pr+qr-6p-2q-4\min(p,q))(r-2)/2\\
		&=& (-(p-2)(q-2)(r-5)-(p-2)(r-5)-(q-2)(r-5)\\	
		&&-(p-3)(q-3)-4\min(p,q)+9)(r-2)/2
		\end{eqnarray*}

	This last expression is clearly negative, for $r\ge 5$. 

	Since we assumed $p, q\le r$, we are left with the cases $(p,q,r) = (3,3,4), (3,4,4)$ and $(4,4,4)$. 
	We check directly that in the cases $(3,3,4)$ and $(4,4,4)$ the expression is actually negative.
	In the case $(3,4,4)$, the expression is equal to~$+2$. Going through all cases of Lemmas~\ref{L:EasyCase} and~\ref{L:HardCase}, 
	one checks that $\Dpqr \lk(a^ib^j, a^{i'}b^{j'})$ is in this case always strictly smaller than $-1$, 
	so that $\Dpqr\lk((a^{p-1}b)^ka^ib^j, (a^{p-1}b)^{k'}a^{i'}b^{j'})$ is always negative in this case. 
	This concludes the proof.
	\end{itemize}
	\vspace{-5mm}
\end{proof}

%%%%%%%%%%%%%%%
\subsection{Some additional computations}

The other needed lemmas are easier.

\begin{lemma}
	\label{L:Oppose}
	For $k, k'\ge 1$, we have $\lk((a^{p-1}b)^ka^ib^j,(ab^{q-1})^{k'}a^{i'}b^{j'}) <0$.
\end{lemma}

\begin{proof}
	By Lemma~\ref{L:Cutting}, we have $\lk((a^{p-1}b)^ka^ib^j,(ab^{q-1})^{k'}a^{i'}b^{j'}) \le k\lk(a^{p-1}b,(ab^{q-1})^{k'}a^{i'}b^{j'}) + \lk(a^ib^j,(ab^{q-1})^{k'}a^{i'}b^{j'})$.
	By Lemma~\ref{L:wijk} the second term is negative, so we only have to show that $\lk(a^{p-1}b,(ab^{q-1})^{k'}a^{i'}b^{j'})$ is negative.
	By Lemma~\ref{L:Cutting}, we have 
	\begin{eqnarray*}
		\lk(a^{p-1}b,(ab^{q-1})^{k'}a^{i'}b^{j'}) &\le& \lk(a^{p-1}b,(ab^{q-1})^{k'}) + \lk(a^{p-1}b,a^{i'}b^{j'})\\
		&=& k'(-pq+p+q) + (qi'-pj') \\ 
		&\le& -pq+p+q + qi'-pj' 
		= q(i'-p+1)+p(1-j'). 
	\end{eqnarray*}
	The latter expression is always negative.
\end{proof}

\begin{lemma}
	\label{L:Mixed2}
	For $k,l, k',l'\ge 1$, we have $\lk((a^{p-1}b)^k(a^{p-1}b)^{l},(ab^{q-1})^{k'}(a^{p-1}b)^{l'}) <0$.
\end{lemma}

\begin{proof}
	Using the superidditivity Lemma~\ref{L:Cutting} and the formula $\Dpqr\lk(a^{p-1}b,a^ib^j) = qi-pj$, we get 
	$\Dpqr\lk((a^{p-1}b)^k(a^{p-1}b)^{l},(ab^{q-1})^{k'}(a^{p-1}b)^{l'}) \le (pq-p-q)(k-l)(k'-l')$, which is not always negative. 	
	So we need to add more precise information as we did with Lemma~\ref{L:Lkijk}. 
	Namely, when following the orbit with code $(a^{p-1}b)^k(a^{p-1}b)^{l}$, we roughly follow $a^{p-1}b$ for $k$ iterations and then $ab^{q-1}$ for $l$ iterations. 
	But when going from the first $k$ blocks to the next $l$ blocks, the orbit enters the intermediate zone, and actually crosses $\min(l,l')$ arcs of $(a^{p-1}b)^{k'}(a^{p-1}b)^{l'}$, namely those arcs that correspond to the last $b$ in the first $\min(l,l')$ blocks of $(ab^{q-1})^{l'}$. 
	Similarly, when the orbit goes from the the part $(ab^{q-1})^l$ to $(a^{p-1}b)^k$, the corresponding orbit crosses $\min(k,k')$ arcs of $(a^{p-1}b)^{k'}(a^{p-1}b)^{l'}$. 
	In this counting, one crossing is counted twice. 
	Thus the additional contribution to~$\lk((a^{p-1}b)^k(a^{p-1}b)^{l},$ $(ab^{q-1})^{k'}(a^{p-1}b)^{l'})$ is $\min(k,k')+\min(l,l')-1$. 
	Then we have $\Dpqr\lk((a^{p-1}b)^k(a^{p-1}b)^{l},$ $(ab^{q-1})^{k'}(a^{p-1}b)^{l'}) = (pq-p-q)(k-l)(k'-l') +(-pqr+pq+qr+pr)(\min(k,k')+\min(l,l')-1)$.
	
	The first term is positive if $k-l$ and $k'-l'$ are of the same sign. Without loss of generality we then assume $k\ge l, k'\ge l'$ and $k\ge k'$. 
	Depending whether $l\ge l'$ or not, we have 
	$\Dpqr\lk((a^{p-1}b)^k(a^{p-1}b)^{l},(ab^{q-1})^{k'}(a^{p-1}b)^{l'}) = (pq-p-q)(k-l)(k'-l') +(-pqr+pq+qr+pr)(k'+l'-1)$ or $(pq-p-q)(k-l)(k'-l') +(-pqr+pq+qr+pr)(k'+l-1)$.
	In both cases, the coefficients of~$k, k', l$ and $l'$ are respectively positive, positive, negative and negative, so these expressions are maximal for $k=k'=\frac {r-2}2, l= l'=1$. 
	They are then both equal to~$(pq-p-q)(\frac{r-4}2)^2 +(-pqr+pq+qr+pr)\frac{r-2}2$.
	The latter expression is negative for $p, q\ge 3, r\ge 4$.
\end{proof}

\begin{lemma}
	\label{L:Mixed}
	For $k\ge 0, k',l'\ge 1$, we have $\lk((a^{p-1}b)^ka^ib^j,(a^{p-1}b)^{k'}(ab^{q-1})^{l'}) <0$.
\end{lemma}

\begin{proof}
	First assume~$k< k'$.
	Lemma~\ref{L:Cutting} then yields~$\cro((a^{p-1}b)^{k}a^ib^j,(a^{p-1}b)^{k'}(ab^{q-1})^{l'}) \le 2kk'(p-1)+2kl+2ik'+2l'$.
	As in the previous proof, there is an additional term $2k+2$ due to the fact that $b^ja^i$ inserts itself in the last block of $(a^{p-1}b)^{k'}(ab^{q-1})^{l'}$, 
	so that we have $\cro((a^{p-1}b)^{k}a^ib^j,(a^{p-1}b)^{k'}(ab^{q-1})^{l'}) = 2kk'(p-1)+2kl+2ik'+2l'+2k+2$. 
	We then obtain $\Dpqr\lk(a^ib^j,(a^{p-1}b)^{k'}(ab^{q-1})^{l'}) = (k'-l')(k(pq-p-q)+qi-pj)-(k+1)\Dpqr$. 
	The latter expression is linear in $i,j,k,k',l'$, and it is maximal for $i=p-1, j= 1, k=0, k'=\frac{r-2}2$ and $l=1$. 
	One checks that even in this case, it is negative.

	Now assume $k\ge k'$. 
	In this case the additional term is only $2k'$. 
	The we have $\cro((a^{p-1}b)^{k}a^ib^j,(a^{p-1}b)^{k'}(ab^{q-1})^{l'}) = 2kk'(p-1)+2kl+2ik'+2l'+2k'$. 
	and thus obtain $\Dpqr\lk(a^ib^j,(a^{p-1}b)^{k'}(ab^{q-1})^{l'}) = (k'-l')(k(pq-p-q)+qi-pj)-k'\Dpqr$. 
	This expression is linear in $i,j,k,k',l'$, and it is maximal for $i=p-1, j=1, k=\frac{r-2}2, k'= 1$ and $l'=1$. 
	Even in this case it is negative.
\end{proof}

%%%%%%%%%%%%%%%
\subsection{Proof of Theorem~\ref{T:Main} in the case $p=2$} 
\label{S:p=2}

First remark that, since the orbifold~$\Sigma_{2,2q',r}$ has an order 2-covering by the orbifold~$\Sigma_{q',q',r}$, left-handedness of the geodesic flow on $\U\Sigma_{2,2q',r}$ can be deduced from left-handedness of the geodesic flow on~$\U\Sigma_{q',q',r}$ (see~\cite[\S~2.4]{Pierre}). Therefore the only case not covered by the analysis when $p\ge 3$ is the case that $q$ and $r$ are odd. 

Lemma~\ref{L:Cutting} being also valid in the case $p=2$, it is enough to compute linking numbers between extremal orbits. 
The extremal orbits of Lemma~\ref{L:Extremal} have to be replaced by the periodic orbits whose codes are of the form 
$ab^i(ab)^k$ or $ab^i(ab^{q-1})^k$ with $(i,k)\in[\![2,q-2]\!]\times[\![0,\frac{r-3}2]\!]$, or $(ab)^k(ab^{q-1})^l$ for $(k,l)$ in~$[\![1,\frac{r-3}2]\!]\times[\![1,\frac{r-3}2]\!]$.
The technics of Lemma~\ref{L:wijk}, \ref{L:Oppose}, \ref{L:Mixed2} and \ref{L:Mixed} can then be copied to prove that all linking numbers of pairs of extremal orbits are always negative (there are four cases again).
Rather than detailing these proofs here, we refer to our Sage worksheet\footnote{\url{https://www-fourier.ujf-grenoble.fr/~dehornop/maths/ComputationsLinkingTpqr.sws
}} for numerical evidence.
Also, the remark of \S~\ref{S:Proof} is valid here: our computations give a proof for all cases that we have computer-checked, here $q\le10, r\le14$.

%%%%%%%%%%%%%%%%%%%%%%%%%%%%%%%%%%%%
\section{Conclusion: the quest of a Gauss linking form}

A classical way of computing linking numbers is \emph{via} a linking form (see \S~\ref{S:LeftHanded}). 
By Ghys Theorem~\cite{GhysLeftHanded} a vector field~$X$ on a 3-manifold~$M$ induces a left-handed flow if and only if there exists a linking form~$\omega_M$ that is everywhere negative along~$X$, \emph{i.e.}, such that $\omega_M(X(x), X(y))<0$ for all $x, y$ in~$M$. 
Theorem~\ref{T:Main} then implies the existence of such a negative linking form on~$\U\Spqr$ for every hyperbolic 3-conic 2-sphere~$\Spqr$. 
It is then natural to wonder whether such a form is canonical or easy-to-construct. 
For example, can it be constructed using propagators~\cite{Lescop}? or by integrating a canonical kernel~\cite{Kuperberg, DTG}?
However, the fact that such a negative Gauss form does not exist on~$\U\Hy/G_{p_1, \dots, p_n}$ for $n\ge 4$ could be discouraging.

Very recently the quest of such an explicit form has been carried out on~$\U\TT^2$ by Adrien Boulanger~\cite{Boulanger}. 
With his formulas it is then obvious that (almost) every null-homologous pair of collections of periodic orbits of the geodesic flow on~$\U\TT^2$ have negative linking number, so that this flow is (almost) left-handed. 
This gives a proof of~\cite[Thm B]{Pierre} that is more natural, as well as some hope to exhibit a natural Gauss form in the higher-genus case.
The \emph{almost} means that actually some linking numbers are zero, so that some peculiar collections do not form fibered links.

%%%%%%%%%%%%%%%%%%%%%%%%%%%%%%%%%%%%

\bibliographystyle{siam}

\end{document}